\documentclass{article}
\pdfoutput=1
\usepackage[a4paper]{geometry}
\usepackage{amssymb,amsmath,amsthm,mathtools,mathrsfs,float}
\usepackage{tgtermes,enumitem}
\usepackage{textcomp}
\usepackage{makecell}
\usepackage[justification=centering]{caption}
\usepackage{multirow,booktabs,subfigure}
\usepackage{color}
\usepackage{hyperref}
\newtheorem{defn}{Definition}[section]

\newtheorem{lem}{Lemma}[section]
\newtheorem{thm}{Theorem}[section]
\newtheorem{rem}{Remark}[section]

\title{Semidefinite Programming Approximation for a Matrix Optimization Problem over an Uncertain Linear System \footnote{~Fang's research was supported by the Walter Clark Professor Endowment of the North Carolina State University. Xing's research was supported by the National Natural Science Foundation of China (Grant No. 11771243).}
}
\author{
Jintao Xu \thanks{~Department of Mathematical Sciences, Tsinghua University, Beijing 100084, China. \newline Email: xujt19@mails.tsinghua.edu.cn}
\and
Shu-Cherng Fang \thanks{~Edward P. Fitts Department of Industrial and Systems Engineering, North Carolina State University, Raleigh, NC 27606, USA. \newline Email: fang@ncsu.edu}
\and
Wenxun Xing\thanks{~Department of Mathematical Sciences, Tsinghua University, Beijing 100084, China. \newline Email: wxing@tsinghua.edu.cn}}
\date{}

\begin{document}
\maketitle
\begin{abstract}
A matrix optimization problem over an uncertain linear system on finite horizon (abbreviated as MOPUL) is studied, in which the uncertain transition matrix is regarded as a decision variable. This problem is in general NP-hard. By using the given reference values of system outputs at each stage, we develop a polynomial-time solvable semidefinite programming (SDP) approximation model for the problem. The upper bound of the cumulative error between reference outputs and the optimal outputs of the approximation model is theoretically analyzed. Two special cases associated with specific applications are considered. The quality of the SDP approximate solutions in terms of feasibility and optimality is also analyzed. Results of numerical experiments are presented to show the influences of perturbed noises at reference outputs and control levels on the performance of SDP approximation.
\end{abstract}
\section{Introduction}\label{section:1}
The discrete-time uncertain linear system is widely studied in control theory \cite{Kothare1996,Cohen1998,Cuzzola2002,Duan2006,Heemels2010,Cairano2016,Tripathy2017}.
The uncertainty may come from the uncertain parameter matrix associated with the given constraint set such as a convex hull \cite{Kothare1996,Cuzzola2002,Duan2006,Heemels2010,Blanco2010} or other settings \cite{Cohen1998,Tripathy2017}.
Robust optimization models are often adopted to deal with the uncertain parameters \cite{Kothare1996,Wan2003,Li2010,Parsi2022}.
However, in many scenarios such as the linear model predictive control (MPC) for optimal tracking \cite{Camacho2007,Alessio2009}, COVID-19 pandemic optimal control \cite{Xu2023}, Markov chain estimation, and enterprise input-output analysis \cite{Liu2020} (to be described in Section 2), we face a new class of matrix optimization problems that regard the uncertain transition matrix as a decision variable. In this paper, we consider the following matrix optimization problem over an uncertain linear system on finite horizon (\ref{MOPUL}):
\begin{align}
\begin{split}\label{MOPUL}
\min_{A, U, \omega}&~~
\lambda_{1}f_{1}\left(A\right)+\lambda_{2}f_{2}\left(U\right)+\lambda_{3}f_{3}\left(\omega\right)\nonumber\\
{\rm s.t.}
&~~x_{t}=Ax_{t-1}+Bu_{t-1},~~t=1, 2, \ldots, N,\\
&~~y_{t}=Cx_{t},~~t=0, 1, \ldots, N,\\
&~~\sum_{t=1}^{N}\left\Vert y_{t}-r_{t}\right\Vert_{2}\leq \omega,\\
&~~\left(A, U, \omega\right)\in\mathcal{S},
\end{split}\tag{MOPUL}
\end{align}
where $A\in\mathbb{R}^{n\times n}$, $U\coloneqq(u_{0}, u_{1}, \dots, u_{N-1})\in\mathbb{R}^{m\times N}$ and $\omega\in\mathbb{R}_+$ are decision variables, $\{\lambda_i\}_{i=1}^3\subseteq\mathbb{R}_+$, $N\in\mathbb{N}_+$, $B\in\mathbb{R}^{n\times m}$, $C\in\mathbb{R}^{p\times n}$, $\{r_t\}_{t=1}^N\subseteq\mathbb{R}^p$ and $\mathcal{S}\subseteq\mathbb{R}^{n\times n}\times\mathbb{R}^{m\times N}\times\mathbb{R}_+$ are given. $\{f_i(\cdot)\}_{i=1}^3$ are assumed to be semidefinite representable (SD representable) functions \cite{BenTal2001,Xing2020}, and $\mathcal{S}$ is assumed to be an SD representable set \cite{BenTal2001,Xing2020}. In addition, $C$ is assumed to be of full column rank.

In problem (\ref{MOPUL}), $f_1(A)$, $f_2(U)$ and $f_3(\omega)$ are the given objective functions of decision variables $A$, $U$ and $\omega$ with given weights of $\lambda_1$, $\lambda_2$ and $\lambda_3$, respectively. For example, $f_1(A)=\Vert A-A^{\rm r}\Vert_F$, where $A^{\rm r}$ is a given reference matrix, $f_2(U)=\sum_{t=1}^{N-1}\Vert u_t-u_{t-1}\Vert_2^2$ as in \cite{Alessio2009,Camacho2007}, and $f_3(\omega)=\omega$. Moreover, $A$, $B$, $C$ and $U$ are the transition matrix, fixed parameter matrices and control, respectively, of the discrete-time linear system on a finite horizon $N$ described by the first two constraints. $x_t\in\mathbb{R}^n$ is the $n$ dimensional system state with a given initial state $x_0$, and $y_{t}\in\mathbb{R}^p$ is the $p$ dimensional system output at the $t$th stage, $t=1, 2, \ldots, N$. $r_{t}$ carries the reference value of the system output $y_t$ for each $t=1, 2, \ldots, N$. $\omega$ is the control level/threshold of the cumulative error between $\{y_t\}_{t=1}^N$ and $\{r_t\}_{t=1}^N$.

The first two constraints are commonly seen in the control theory of discrete-time finite horizon linear systems. The difference is that the transition matrix $A$ is a decision variable in (\ref{MOPUL}) and an uncertain parameter in control theory.
The third constraint restricts the cumulative error between the system outputs and their reference values within a control level $\omega$. When $f_3(\omega) = \omega$, the cumulative error constraint can be lifted to the objective function. Other restrictions on the decision variables are contained in the set $\mathcal{S}$ as the fourth constraint, and entanglement of decision variables is allowed in it. In Section \ref{section:2}, we show that the (\ref{MOPUL}) model is widely applicable.

Notice that the first three constraints of (\ref{MOPUL}) are multivariate polynomial constraints. Since a multivariate polynomial optimization problem in general is NP-hard \cite{Nesterov2000}, we know (\ref{MOPUL}) is generally computationally intractable.
On the other hand, semidefinite programs (SDP) are polynomial-time solvable \cite{BenTal2001, Nesterov1994, Terlaky1996} with many successful applications to the control theory \cite{DAVIDD.2001, Balakrishnan2003,Blanco2010,Tanaka2018, Bujarbaruah2020}, multiple-input multiple-output (MIMO) analysis \cite{CHENG2019,Mobasher2007}, combinatorial optimization problems \cite{Goemans1995,Gaar2020,Han2002,Guimarildeaes2020}, and portfolio selection problems \cite{Ghaoui2003,Chen2011}. SDP solvers such as SeDuMi (\href{https://sedumi.ie.lehigh.edu}{https://sedumi.ie.lehigh.edu}), MOSEK (\href{https://www.mosek.com}{https://www.mosek.com}) and DSDP (\href{https://www.mcs.anl.gov/hs/software/DSDP/}{https://www.mcs.anl.gov/hs/software/DSDP/}) are readily available.

The first contribution of this paper is to construct an SDP approximation model for (\ref{MOPUL}). Notice that $r_{t}$ is regarded as the reference value of $y_{t}$. We can use $C^{\dag}r_{t}$ to approximate $x_t$, where $C^{\dag}$ denotes the Moore-Penrose inverse of matrix $C$. Similar to \cite{Xu2023}, we can replace the first constraint by $x_{t}=AC^{\dag}r_{t-1}+Bu_{t-1}, t=1, 2, \ldots, N$, to reformulate (\ref{MOPUL}) as an SDP approximation model.

The second contribution of this paper is to provide a theoretical analysis of the quality of SDP approximate solutions in terms of the feasibility and optimality. For an SDP approximate solution $(A^{{\rm a}*}, U^{{\rm a}*}, \omega^{{\rm a}*})$ and consequently output values of $y_{t}$ at stage $t$ of the linear system, an upper bound of the cumulative error $\sum_{t=1}^{N}\Vert y_{t}-r_{t}\Vert_{2}$ corresponding to $(A^{{\rm a}*},U^{{\rm a}*},\omega^{{\rm a}*})$ is provided in Theorem \ref{theorem:2} for general setting. Moreover, the feasibility of an SDP approximate solution to (\ref{MOPUL}) with respect to a fixed control level is guaranteed in Theorem \ref{theorem:3}. Motivated by the application problems, two special cases of (\ref{MOPUL}) with SDP approximations concerning two settings of $(\{\lambda_{i}\}_{i=1}^{3}, \{f_i\}_{i=1}^{3}, \mathcal{S})$ are considered in Subsection \ref{subsection:3.3} for better theoretical estimations on the optimal objective values.

The third contribution of this paper is to show the influences of perturbed noise levels at reference outputs and control levels on the performance of the SDP approximation model through numerical experiments. Equipped with accurate reference outputs and proper control levels, SDP approximation performs really well numerically.

The rest of the paper is organized as follows. Some specific applications of (\ref{MOPUL}) are introduced in Section \ref{section:2}. In Section \ref{section:3}, an SDP approximation model is constructed, and theoretic analysis of its performance is provided. Numerical results are reported in Section \ref{section:4} and some concluding remarks are made in Section \ref{section:5}.

\textbf{Notations.} Throughout the paper, $\mathbb{R}^{n}$, $\mathbb{R}_{+}^{n}$, $\mathbb{R}^{m\times n}$, and $\mathbb{N}_{+}$ denote the sets of real $n$-dimensional vectors, nonnegative vectors, $m\times n$ matrices, and positive integers, respectively. $\textbf{S}^{n}$, $\textbf{S}_{+}^{n}$, and $\textbf{S}_{++}^{n}$ denote the sets of real $n\times n$ symmetric, positive semidefinite ($X\succeq0$), and positive definite matrices, respectively.
$X^{\dag}$ denotes the Moore-Penrose inverse of $X$. $\Vert x\Vert_{2}=(\sum_{i=1}^{n}x_{i}^{2})^{\frac{1}{2}}$ and $\Vert x\Vert_{Q}=\sqrt{x^{T}Qx}$, where $Q\in\textbf{S}_{++}^{n}$. $\Vert X\Vert_{F}$, $\Vert X\Vert_{2}$, and $\Vert X\Vert_{*}$ denote the Frobenius norm, the spectral norm which is equal to the maximum singular value of $X$, and the nuclear norm which is equal to the sum of all singular values of matrix $X$, respectively.
$O$ and $I$ denote the matrix of all zeros and the unit matrix whose sizes vary from the context, respectively. $\boldsymbol{0}$ and $\boldsymbol{1}$ denote the column vector of all zeros and ones whose sizes vary from the context, respectively.

\section{Applications}\label{section:2}
In this section, we present four specific applications of problem (\ref{MOPUL}). Their special structures and the quality of the corresponding SDP approximate solutions in terms of feasibility and optimality will be further investigated in Sections \ref{section:3} and \ref{section:4}.
\subsection{Linear model predictive control for optimal tracking}\label{subsection:2.1}
Model predictive control (MPC) is a class of optimal control strategies, in which the optimizer determines control signals and the model predicts outputs \cite{Camacho2007}.
Referring to equation (18) in \cite{Alessio2009}, equation (2.5) in \cite{Camacho2007}, and related discussions therein, an optimal tracking problem over an uncertain linear system goes in the following form:

\begin{align}
\begin{split}\label{O-MPC}
\min_{A, U}&~~\sum_{t=1}^{N}\left\Vert y_{t}-r_{t}\right\Vert_{2}+\lambda\sum_{t=1}^{N-1}\left\Vert u_{t}-u_{t-1}\right\Vert_{2}\\
{\rm s.t.}
&~~x_{t}=Ax_{t-1}+Bu_{t-1},~~t=1, 2, \ldots, N,\\
&~~y_{t}=Cx_{t},~~t=0, 1, \ldots, N,\\
&~~(A, U)\in\mathcal{S}_{\rm\scriptscriptstyle MPC},
\end{split}\tag{O-MPC}
\end{align}
where the transition matrix $A\in\mathbb{R}^{n\times n}$ and control $U\coloneqq(u_{0}, u_{1}, \ldots, u_{N-1})\in\mathbb{R}^{m\times N}$ are decision variables, system horizon $N\in\mathbb{N}_{+}$, parameter $\lambda\geq0$, $\{x_t\}_{t=0}^N\subseteq\mathbb{R}^n$ are $n$ dimensional system states with a given initial state $x_0$, $\{y_t\}_{t=0}^N\subseteq\mathbb{R}^p$ are $p$ dimensional system outputs, $B\in\mathbb{R}^{n\times m}$ and $C\in\mathbb{R}^{p\times n}$ are given system parameter matrices, and $\{r_{t}\}_{t=1}^{N}\subseteq\mathbb{R}^p$ are given reference signals. The cumulative error $\sum_{t=1}^{N}\Vert y_{t}-r_{t}\Vert_{2}$ enforces system outputs to track the given reference signals, and control efforts $\sum_{t=1}^{N-1}\left\Vert u_{t}-u_{t-1}\right\Vert_{2}$ are penalized for variations.
SD representable set $\mathcal{S}_{\rm\scriptscriptstyle MPC}=\mathcal{S}_{\rm\scriptscriptstyle UC}\cap~\mathcal{S}_{\rm\scriptscriptstyle AR}\subseteq\mathbb{R}^{n\times n}\times\mathbb{R}^{m\times N}$, where $\mathcal{S}_{\rm\scriptscriptstyle UC}$ is the uncertainty set of linear system and $\mathcal{S}_{\rm\scriptscriptstyle AR}$ is composed of additional restrictions on $(A, U)$.
Examples of $\mathcal{S}_{\rm\scriptscriptstyle UC}$ include
$\{{\rm constant}~A\}\times\mathbb{R}^{m\times N}$ for the linear system with a constant transition matrix \cite{Bemporad2000} and $\{\sum_{i=1}^{k}\theta^i A^i, (\theta^{1}, \theta^{2}, \ldots, \theta^{k})^{\mathrm{T}}\in\mathbb{R}_{+}^{k}, \sum_{i=1}^{k}\theta^{i}=1\}\times\mathbb{R}^{m\times N}$ with given $k$ matrices $\{A^i\}_{i=1}^k$ for the uncertain linear system with transition matrix in the polytopic uncertainty set \cite{Cairano2016}.
Examples of $\mathcal{S}_{\rm\scriptscriptstyle AR}$ include
$\mathbb{R}^{n\times n}\times\{U|\alpha_{1}\leq u_{t}-u_{t-1}\leq\alpha_{2}\}$ with $\alpha_{1}, \alpha_{2}\in\mathbb{R}^m$ and component-wise inequalities  \cite{Alessio2009}.
Notice that (\ref{O-MPC}) is a special case of (\ref{MOPUL}) by setting $\lambda_{1}=0$, $\lambda_{2}=\lambda$, $\lambda_{3}=1$, $f_2(U)=\sum_{t=1}^{N-1}\left\Vert u_{t}-u_{t-1}\right\Vert_{2}$, $f_3(\omega)=\omega$, and $\mathcal{S}=\mathcal{S}_{\rm\scriptscriptstyle MPC}\times\mathbb{R}_+$.

\subsection{COVID-19 pandemic optimal control model}\label{subsection:2.2}
To realize an effective prevention and control for the COVID-19 pandemic, we can construct the so-called ``susceptible-asymptomatic infected-symptomatic infected-removed optimal control'' model as below by dividing the total population into 4 groups of susceptible (S), asymptomatic infected (I$_\text{a}$), symptomatic infected (I$_\text{s}$), and removed (R).
\begin{align}
\begin{split}\label{O-COVID}
\min_{A, U}&~~\sum_{t=1}^{N}\left\Vert x_{t}-r_{t}\right\Vert_{2}\\
{\rm s.t.}&~~
x_{t}=Ax_{t-1}+u_{t-1},~~t=1, 2,\ldots, N,\\
&~~\left(A, U\right)\in\mathcal{S}_{\rm\scriptscriptstyle COVID},
\end{split}\tag{O-COVID}
\end{align}
where the transmission matrix $A\in\mathbb{R}^{4\times4}$ and the exit and entry control $U\coloneqq(u_{0}, u_{1}, \ldots, u_{N-1})\in\mathbb{R}^{4\times N}$ are decision variables, $N\in\mathbb{N}_+$ is the duration of COVID-19 transmission studied in (\ref{O-COVID}),
$\{x_{t}\coloneqq(x_t^{\rm\scriptscriptstyle S}, x_t^{\rm\scriptscriptstyle I_a}, x_t^{\rm\scriptscriptstyle I_s}, x_t^{\rm\scriptscriptstyle R})^{\mathrm T}\}_{t=0}^N$ and $\{u_{t}\coloneqq(u_t^{\rm\scriptscriptstyle S}, u_t^{\rm\scriptscriptstyle I_a}, u_t^{\rm\scriptscriptstyle I_s}, u_t^{\rm\scriptscriptstyle R})^{\mathrm T}\}_{t=0}^N\subseteq\mathbb{R}^4$ are the numbers of individuals in each group and their variations through exit and entry, respectively, and $\{r_{t}\coloneqq(r_t^{\rm\scriptscriptstyle S}, r_t^{\rm\scriptscriptstyle I_a}, r_t^{\rm\scriptscriptstyle I_s}, r_t^{\rm\scriptscriptstyle R})^{\mathrm T}\}_{t=1}^N\subseteq\mathbb{R}^4$ are the expected numbers of individuals in each group.
Then COVID-19 transmission is $x_{t}=Ax_{t-1}+u_{t-1}$ in (\ref{O-COVID}). Additional constraints on $(A, U)$ are contained in the SD representable constraint set $\mathcal{S}_{\rm\scriptscriptstyle COVID}\subseteq\mathbb{R}^{4\times 4}\times\mathbb{R}^{4\times N}$.
To realize the target $\{r_t\}_{t=1}^{N}$ estimated by the medical facilities, the transmission matrix $A$ and the exit and entry control $U$ are determined in (\ref{O-COVID}) through the minimization of $\sum_{t=1}^{N}\Vert x_{t}-r_{t}\Vert_{2}$.
Notice that (\ref{O-COVID}) is a special case of (\ref{MOPUL}) by setting $\lambda_1=\lambda_2=0, \lambda_3=1$, $f_3(\omega)=\omega$, $B=C=I$, and $\mathcal{S}=\mathcal{S}_{\rm\scriptscriptstyle COVID}\times\mathbb{R}_+$.

\subsection{Markov chains Estimation}\label{subsection:2.3}
Let $\{X_{t}\}_{t\geq0}$ be a homogeneous Markov chain on states $\{s_{i}\}_{i=1}^{m}$ with an unknown low-rank transition matrix $P=(p_{ij}\coloneqq\mathbb{P}\left(X_{t}=s_{j}|X_{t-1}=s_{i}\right))_{m\times m}\in\mathbb{R}^{m\times m}$, which implies a latent low-dimensionality structure \cite{Zhu2021}. We can construct an optimization model for the Markov chains estimation with a low-rank demand as the following:
\begin{align}
\begin{split}\label{O-Markov}
\min_{P}&~~\sum_{t=1}^{N}\left\Vert \pi_{t}-r_{t}\right\Vert_{2}\\
{\rm s.t.}
&~~\pi_{t}=P\pi_{t-1},~~t=1, 2, \ldots, N,\\
&~~P\in\mathcal{S}_{\rm\scriptscriptstyle Markov},
\end{split}\tag{O-Markov}
\end{align}
where the transition matrix $P$ is a decision variable, observation horizon $N\in\mathbb{N}_+$, probability distributions $$\pi_{t}\coloneqq\left(\mathbb{P}\left(X_{t}=s_{1}\right), \mathbb{P}\left(X_{t}=s_{2}\right), \ldots, \mathbb{P}\left(X_{t}=s_{m}\right)\right)^{\mathrm{T}}\in\mathbb{R}^{m}, t=0, 1, \ldots, N,$$
the $i$th component of $r_t$, i.e. $(r_{t})_{i}$ is an observed frequency of the event $\{X_{t}=s_{i}\}$, for $i=1, 2, \ldots, m$, $t=1, 2, \ldots, N$, and
\begin{align*}
\mathcal{S}_{\rm\scriptscriptstyle Markov}\coloneqq\left\{P=\left(p_{ij}\right)_{m\times m}\in\mathbb{R}^{m\times m}\left|\begin{array}{ll}
& p_{ij}\geq0, i, j=1, 2, \ldots, m,\\
& \sum_{i=1}^{m}p_{ij}=1, j=1, 2, \ldots, m,\\
& \left\Vert P\right\Vert_{*}\leq\alpha,\\
& \text{and subject to a finite number of linear inequality}\\
& \text{constraints on}~P.
\end{array}
\right.\right\},
\end{align*}
in which $\Vert\cdot\Vert_*$ denotes the nuclear norm and $0<\alpha<m$.
Different from Zhang and Wang \cite{Zhang2020}, Li et al. \cite{Li2018}, and Zhu et al. \cite{Zhu2021} of using the information of event $\{X_{t-1}=s_{i}, X_{t}=s_{j}\}$, (\ref{O-Markov}) estimates the low-rank transition matrix through frequency approximation of event $\{X_{t}=s_{i}\}$.
The low-rank demand is enforced by the nuclear norm constraint
in $\mathcal{S}_{\rm\scriptscriptstyle Markov}$. Notice that (\ref{O-Markov}) is a special case of (\ref{MOPUL}) by setting $\lambda_1=\lambda_2=0, \lambda_3=1$, $f_3(\omega)=\omega$, $B=O, C=I$, and $\mathcal{S}=\mathcal{S}_{\rm\scriptscriptstyle Markov}\times \mathbb{R}^{m\times N}\times\mathbb{R}_+$, where $O$ is the matrix of all zeros.

\subsection{Multi-stage enterprise input-output problems}\label{subsection:2.4}
Input-output analysis is a framework describing and analyzing input (consumption) and output (production) activities and their relations in an economy \cite{Miller2009,Liu2020}. Referring to \cite{Liu2020}, considering an enterprise production with $m_1$ self-made and $m_2$ out-sourced products, the following multi-stage enterprise input-output optimization problem can be constructed to realize the given expected output values of enterprise production by controlling production technologies and purchase-sale plans.
\begin{align}
\begin{split}\label{O-IN/OUTPUT1}
\min_{A, U}&~~\sum_{t=1}^{N}\left\Vert x_{t}-r_{t}\right\Vert_{2}\\
{\rm s.t.}
&~~x_{t}=Ax_{t-1}+u_{t-1},~~t=1, 2, \ldots, N,\\
&~~A\in\mathcal{S}_{\rm\scriptscriptstyle IO},
\end{split}\tag{O-IN/OUTPUT1}
\end{align}
where the production technology matrix $A\in\mathbb{R}^{(m_1+m_2)\times (m_1+m_2)}$ with structure described in (\ref{S_IO}) and purchase-sale control $U\coloneqq(u_{0}, u_{1}, \ldots, u_{N-1})\in\mathbb{R}^{(m_1+m_2)\times N}$ are decision variables, $N\in\mathbb{N}_+$ is the duration of enterprise production, $\{x_{t}\}_{t=1}^N\subseteq\mathbb{R}^{m_1+m_2}$ are the production output values of $m_1+m_2$ products in which $(x_{t})_{i}$ is the output value of the $i$th self-made product, $i=1, 2, \ldots, m_1$, and $(x_t)_{m_1+j}$ is the output value of the $j$th out-sourced product at each stage, $j=1, 2, \ldots, m_2$, $\{u_t\}_{t=0}^{N-1}\subseteq\mathbb{R}^{m_1+m_2}$ are the purchase-sale values of $m_1+m_2$ products at each stage, $\{r_t\}_{t=1}^{N}\subseteq\mathbb{R}^{m_1+m_2}$ are given expected output values as the references for $\{x_t\}_{t=1}^N$, and constraint set
\begin{align}\label{S_IO}
\mathcal{S}_{\rm\scriptscriptstyle IO}\coloneqq
\left\{\begin{pmatrix}
I-G & O\\
-H & I
\end{pmatrix}\in\mathbb{R}^{(m_1+m_2)\times (m_1+m_2)}\left|
\begin{array}{l}
\text{subject to a finite number of linear inequality}\\
\text{constraints on $G\in\mathbb{R}^{m_1\times m_1}$ and $H\in\mathbb{R}^{m_2\times m_{1}}$}.\end{array}\right.\right\},
\end{align}
in which $G$ and $H$ are composed of technical coefficients \cite{Liu2020}. Then enterprise production is $x_{t}=Ax_{t-1}+u_{t-1}$ in (\ref{O-IN/OUTPUT1}). The production technology matrix $A$ and purchase-sale control $U$ are determined to realize the expected enterprise output values by minimizing the discrepancy between the system output and the expected output values $\sum_{t=1}^{N}\left\Vert x_{t}-r_{t}\right\Vert_{2}$.
Notice that (\ref{O-IN/OUTPUT1}) is a special case of (\ref{MOPUL}) by setting $\lambda_1=\lambda_2=0, \lambda_3=1$, $f_3(\omega)=\omega$, $B=C=I$, and $\mathcal{S}=\mathcal{S}_{\rm\scriptscriptstyle IO}\times \mathbb{R}^{(m_1+m_2)\times N}\times \mathbb{R}_+$.

When a steady and controllable change of the production technology is preferred within a guaranteed level of
cumulative error, we may consider the following problem:
\begin{align}
\begin{split}\label{O-IN/OUTPUT2}
\min_{A, U}&~~\left\Vert A-A^{\rm r}\right\Vert_{F}\nonumber\\
{\rm s.t.}&~~x_{t}=Ax_{t-1}+u_{t-1},~~t=1, 2, \ldots, N,\\
&~~\sum_{t=1}^{N}\left\Vert x_{t}-r_{t}\right\Vert_{2}\leq\omega,\\
&~~\Vert u_{t}-u_{t}^{\rm r}\Vert_{2}\leq\omega_{t},~~t=0, 1,\ldots, N-1,\\
&~~A\in\mathcal{S}_{\rm\scriptscriptstyle IO},
\end{split}\tag{O-IN/OUTPUT2}
\end{align}
where the production technology matrix $A\in\mathbb{R}^{(m_1+m_2)\times(m_1+m_2)}$ and purchase-sale control $U\coloneqq(u_{0}, u_{1},$\\ $\ldots, u_{N-1})\in\mathbb{R}^{(m_1+m_2)\times N}$ are decision variables, $N\in\mathbb{N}_+$ is the duration of enterprise production, $\{r_t\}_{t=1}^N\subseteq\mathbb{R}^{m_1+m_2}$, $A^{\rm r}\in\mathbb{R}^{(m_1+m_2)\times(m_1+m_2)}$, and $\{u_{t}^{\rm r}\}_{t=0}^{N-1}\subseteq\mathbb{R}^{m_1+m_2}$ are given reference values of $\{x_t\}_{t=1}^N$, $A$, and $\{u_{t}\}_{t=0}^{N-1}$, respectively, and $\omega, \{\omega_{t}\}_{t=0}^{N-1}\subseteq\mathbb{R}_+$ are control levels.
The second constraint guarantees a cumulative precision of the iteration within the control level $\omega$. And the third constraint guarantees a controllable change of the purchase-sale values.
Notice that (\ref{O-IN/OUTPUT2}) is a special case of (\ref{MOPUL}) by setting $\lambda_1=1, \lambda_2=\lambda_3=0$, $f_1(A)=\Vert A-A^{\rm r}\Vert_{F}$, $B=C=I$, and $\mathcal{S}=\mathcal{S}_{\rm\scriptscriptstyle IO}\times\{U|\Vert u_{t}-u_{t}^{\rm r}\Vert_{2}\leq\omega_{t}, t=0, 1,\ldots, N-1\}\times\{{\rm constant}~\omega\}$.

\section{Semidefinite approximation}\label{section:3}
In this section, we explore the SDP relaxation for problem (\ref{MOPUL}).
In Subsection \ref{subsection:3.1}, we discuss the computational intractability of the problem and construct a polynomial-time solvable SDP approximation model in Subsection \ref{subsection:3.2}. The quality of SDP approximate solutions to two specific settings in terms of feasibility and optimality is analyzed in Subsection \ref{subsection:3.3}.
\subsection{Computational intractability}\label{subsection:3.1}
When $\{f_i\}_{i=1}^3$ and $\mathcal{S}$ are SD representable, the computational intractability of (\ref{MOPUL}) mainly comes from the entanglement of decision variables in the first three constraints.
Specifically, combined with the first two constraints, the third constraint of (\ref{MOPUL}) is equivalent to
\begin{align*}
& \sum_{t=1}^{N}\xi_{t}\leq \omega,\\
& \left\Vert y_{t}-r_{t}\right\Vert_{2}^{2}\leq\xi_t^2, 0\leq\xi_{t},~t=1, 2,\ldots, N,
\end{align*}
where
\begin{align}
\begin{split}\label{poly}
\left\Vert y_{t}-r_{t}\right\Vert_{2}^{2}=&x_{0}^{\mathrm{T}}(A^{\mathrm{T}})^{t}C^{\mathrm{T}}CA^{t}x_{0}+2\sum_{j=0}^{t-1}x_{0}^{\mathrm{T}}(A^{\mathrm{T}})^{t}C^{\mathrm{T}}CA^{j}Bu_{t-1-j}\\
&+\sum_{i, j=0}^{t-1}u_{t-1-i}^{\mathrm{T}}B^{\mathrm{T}}(A^{\mathrm{T}})^{i}C^{\mathrm{T}}CA^{j}Bu_{t-1-j}-2x_{0}^{\mathrm{T}}(A^{\mathrm{T}})^{t}C^{\mathrm{T}}r_{t}\\
&-2\sum_{i=0}^{t-1}u_{t-1-i}^{\mathrm{T}}B^{\mathrm{T}}(A^{\mathrm{T}})^{i}C^{\mathrm{T}}r_{t}+r_{t}^{\mathrm{T}}r_{t}.
\end{split}
\end{align}
This is a nonnegative multivariate polynomial of degree $2t$, $t=1, 2,\ldots, N$ over $A$. Thus (\ref{MOPUL}) equivalently contains a series of multivariate polynomial constraints.
Since the problem of minimizing a nonnegative multivariate polynomial of degree higher than or equal to 4 is in general NP-hard \cite{Nesterov2000}, we know (\ref{MOPUL}) is NP-hard.

\subsection{Approximation model}\label{subsection:3.2}
Notice that the vector $r_{t}$ in (\ref{MOPUL}) can be viewed as given reference values of the system output $y_{t}$ at each stage. For (\ref{O-MPC}) in Subsection \ref{subsection:2.1}, it represents the reference signal in the linear control system. For (\ref{O-COVID}) in Subsection \ref{subsection:2.2}, it represents the expected number of individuals. For (\ref{O-Markov}) in Subsection \ref{subsection:2.3}, it represents the observed frequency of certain event. And for (\ref{O-IN/OUTPUT1}) and (\ref{O-IN/OUTPUT2}) in Subsection \ref{subsection:2.4}, it represents the expected output value of enterprise production. In the proposed approximation model, with the similar idea of \cite{Xu2023}, $r_{t}$ is used to decouple the nested iteration of $x_{t}$ to avoid the multivariate polynomial structures in (\ref{poly}). Specifically, we replace the constraint $x_{t}=Ax_{t-1}+Bu_{t-1}$, $t=1, 2, \ldots, N$ in (\ref{MOPUL}) by $x_{t}=AC^{\dag}r_{t-1}+Bu_{t-1}$, $t=1, 2, \ldots, N$ . Then an approximate matrix optimization problem over an uncertain linear system on finite horizon (abbreviated as \ref{AMOPUL}) can be constructed as the following:
\begin{align}
\begin{split}\label{AMOPUL}
\min_{A, U, \omega}&\ \  \lambda_{1}f_{1}\left(A\right)+\lambda_{2}f_{2}\left(U\right)+\lambda_{3}f_{3}\left(\omega\right)\\
{\rm s.t.}
&\ \ r_{0}=Cx_{0},~~x_{0}^{\rm a}=x_{0},\\
&\ \ x_{t}^{\rm a}=AC^{\dag}r_{t-1}+Bu_{t-1},~~t=1, 2,\ldots, N,\\
&\ \ y_{t}^{\rm a}=Cx_{t}^{\rm a},~~t=0, 1,\ldots, N,\\
&\ \ \sum_{t=1}^{N}\left\Vert y_{t}^{\rm a}-r_{t}\right\Vert_{2}\leq \omega,\\
&\ \ (A,U,\omega)\in \mathcal{S},
\end{split}\tag{AMOPUL}
\end{align}
where the transition matrix $A\in\mathbb{R}^{n\times n}$, control $U\coloneqq(u_{0}, u_{1}, \dots, u_{N-1})\in\mathbb{R}^{m\times N}$, and control level/thre-shold $\omega\in\mathbb{R}_+$ are decision variables. To avoid potential confusions in notation, $x_{t}^{\rm a}$ and $y_{t}^{\rm a}$ are used to denote the approximate values of $x_{t}$ and $y_{t}$ in (\ref{MOPUL}), respectively. We call $\sum_{t=1}^{N}\Vert y_{t}^{\rm a}-r_{t}\Vert_{2}$ the approximate cumulative error. The meanings of other notations are the same as in (\ref{MOPUL}).

Definitions and some properties of the SD representability are included below.
\begin{defn}{\textbf{(Semidefinite representable set  \cite{BenTal2001})}}\label{definition:1}\\
A convex set $\mathcal{C}\subseteq\mathbb{R}^{n}$ is called semidefinite representable (SD representable) if there exist
\begin{align*}
L_{i}^{j}\in\textbf{S}^{m_j}, i=0, 1, \ldots, n+d, j=1,2,\dots, J
\end{align*}
such that
\begin{align*}
\mathcal{C}=\left\{x=\left(x_{1}, x_{2}, \ldots, x_{n}\right)^{\mathrm{T}}\in\mathbb{R}^{n}\left|
\begin{array}{ll}
&
L_{0}^{j}+\sum_{i=1}^{n}x_{i}L_{i}^{j}+\sum_{i=1}^{d}u_{i}L_{n+i}^{j}\succeq 0,~~j=1,2\dots,J\\
& \mathrm{\ for\ some\ } u=\left(u_{1}, u_{2}, \ldots, u_{d}\right)^{\mathrm{T}}\in\mathbb{R}^{d}
\end{array}
\right.\right\}.
\end{align*}
\end{defn}
\begin{defn}{\textbf{(Semidefinite representable function \cite{BenTal2001})}}\label{definition:2}\\
A convex function $f: \mathbb{R}^{n}\rightarrow\mathbb{R}\cup\{+\infty\}$ is called semidefinite representable (SD representable) if the set $\{(x, \lambda)\in\mathbb{R}^{n+1}|f(x)\leq \lambda\}$ is SD representable.
\end{defn}
\noindent SD representability of sets is preserved through the set operations of intersection, direct product, affine mapping and its inverse \cite{BenTal2001}. Notice that the matrix norm $\Vert\cdot\Vert_{F}$, $\Vert\cdot\Vert_Q$, $\Vert\cdot\Vert_{2}$, and $\Vert\cdot\Vert_{*}$ used in this paper are all SD representable functions \cite{BenTal2001}.
The next lemma discloses the connections between the SD representability and SDP problems.
\begin{lem}[\cite{BenTal2001}]\label{lemma:1}
A minimization problem $\min_{x}\{c^{\mathrm{T}}x|x\in\cap_{i=1}^{k}\mathcal{C}_{i}\subseteq\mathbb{R}^n\}$ can be equivalently formulated as an SDP problem if $\mathcal{C}_{i}$ is SD representable for $i=1, 2, \ldots, k$.
\end{lem}

In order to construct an equivalent SDP reformulation, we need the next two results.
\begin{lem}\label{lemma:2}
Let $x\in\mathbb{R}^{n}$ and $X=\begin{pmatrix*}
O & x\\
x^{\mathrm{T}} & 0
\end{pmatrix*}\in\textbf{S}^{n+1}$. If $X\in\textbf{S}_{+}^{n+1}$, then $x=\boldsymbol{0}$.
\end{lem}
\begin{proof}
Let $v_{i, n+1}=(0, \ldots, 0, 1, 0, \ldots, 0, 1)^{\mathrm{T}}\in\mathbb{R}^{n+1}$, whose $i$th and $(n+1)$th elements are 1, and $\bar{v}_{i, n+1}=(0, \ldots, 0, -1, 0, \ldots, 0, 1)^{\mathrm{T}}\in\mathbb{R}^{n+1}$, whose $i$th element is $-1$ and $(n+1)$th element is 1, $i=1, 2, \ldots, n$. Since $X\in\mathcal{S}_+^{n+1}$, we know that
\begin{align*}
2x_{i}=v_{i, n+1}^{\mathrm{T}}Xv_{i, n+1}\geq0,~\text{and}~ -2x_{i}=\bar{v}_{i, n+1}^{\mathrm{T}}X\bar{v}_{i, n+1}\geq0, i=1, 2, \ldots, n.
\end{align*}
Therefore $x=\boldsymbol{0}$.
\end{proof}
\begin{lem}{(Schur complement \cite[Theorem 1.12(b) ]{Horn2005}}\label{lemma:3})\\
Let $M\in\textbf{S}^{n}$ be partitioned as
\begin{align*}
M=\begin{pmatrix}
E & F\\
F^{\mathrm{T}} & G
\end{pmatrix},
\end{align*}
where $E\in\textbf{S}^{q}$ is nonsingular with $q\leq n-1$. Then $M\in\textbf{S}_{+}^{n}$ if and only if $E\in\textbf{S}_{++}^{q}$ and $G-F^{\mathrm{T}}E^{-1}F\in\textbf{S}_{+}^{n-q}$.
\end{lem}

\begin{thm}\label{theorem:1}
Under the assumption that $\{f_{i}\}_{i=1}^{3}$ and $\mathcal{S}$ are SD representable, problem (\ref{AMOPUL}) has the following SDP reformulation:
{\small \begin{align*}
\min_{A, U, \omega, \left\{\xi_{t}\right\}_{t=1}^{N}}&~~ \lambda_{1}f_{1}\left(A\right)+\lambda_{2}f_{2}\left(U\right)+\lambda_{3}f_{3}\left(\omega\right)\\
{\rm s.t.}~~~~~&~~
\begin{pmatrix}
\xi_{t}I & CAC^{\dag}r_{t-1}+CBu_{t-1}-r_{t}\\
\left(CAC^{\dag}r_{t-1}+CBu_{t-1}-r_{t}\right)^{\mathrm{T}} & \xi_{t}
\end{pmatrix}\in\textbf{S}_{+}^{p+1},\\
&~~t=1, 2, \ldots, N,\\
&~~\sum_{t=1}^{N}\xi_{t}\leq\omega,\\
&~~(A,U,\omega)\in \mathcal{S},\nonumber
\end{align*}}

\vspace{-1mm}
\noindent where $A\in\mathbb{R}^{n\times n}$, $U\coloneqq(u_{0}, u_{1}, \dots, u_{N-1})\in\mathbb{R}^{m\times N}$, $\omega\in\mathbb{R}_+$, and $\left\{\xi_{t}\right\}_{t=1}^{N}\subseteq\mathbb{R}_+$ are decision variables, and the remaining notations are defined the same as in (\ref{AMOPUL}).
\end{thm}

\begin{proof}
(\ref{AMOPUL}) is equivalent to
\begin{align*}
\min_{A, U, \omega, \left\{\xi_{t}\right\}_{t=1}^{N}}&~~ \lambda_{1}f_{1}\left(A\right)+\lambda_{2}f_{2}\left(U\right)+\lambda_{3}f_{3}\left(\omega\right)\\
{\rm s.t.}~~~~~&~~
\left\Vert CAC^{\dag}r_{t-1}+CBu_{t-1}-r_{t}\right\Vert_{2}^{2}\leq\xi_{t}^{2},~~ \xi_{t}\geq0,~~t=1, 2, \ldots, N,\\
&~~\sum_{t=1}^{N}\xi_{t}\leq \omega,\\
&~~(A,U,\omega)\in \mathcal{S}\nonumber.
\end{align*}
We now prove that the first constraint in the above reformulation is equivalent to
\begin{align*}
\begin{pmatrix}
\xi_{t}I & CAC^{\dag}r_{t-1}+CBu_{t-1}-r_{t}\\
\left(CAC^{\dag}r_{t-1}+CBu_{t-1}-r_{t}\right)^{\mathrm{T}} & \xi_{t}
\end{pmatrix}\in\textbf{S}_{+}^{p+1},~~t=1, 2, \ldots, N.
\end{align*}
For each $t=1, 2, \ldots, N$, if $\xi_{t}=0$, then
\begin{align*}
&\left\Vert CAC^{\dag}r_{t-1}+CBu_{t-1}-r_{t}\right\Vert_{2}^{2}\leq\xi_{t}^{2}, \xi_{t}=0\\
\Longleftrightarrow& ~CAC^{\dag}r_{t-1}+CBu_{t-1}-r_{t}=\boldsymbol{0}\\
\Longleftrightarrow& ~\begin{pmatrix}
O & CAC^{\dag}r_{t-1}+CBu_{t-1}-r_{t}\\
\left(CAC^{\dag}r_{t-1}+CBu_{t-1}-r_{t}\right)^{\mathrm{T}} & 0
\end{pmatrix}\in\textbf{S}_{+}^{p+1}\\
\Longleftrightarrow& ~\begin{pmatrix}
\xi_{t}I & CAC^{\dag}r_{t-1}+CBu_{t-1}-r_{t}\\
\left(CAC^{\dag}r_{t-1}+CBu_{t-1}-r_{t}\right)^{\mathrm{T}} & \xi_{t}
\end{pmatrix}\in\textbf{S}_{+}^{p+1},
\end{align*}
where the second equivalency follows from Lemma \ref{lemma:2}.

If $\xi_{t}>0$, then
\begin{align*}
&\left\Vert CAC^{\dag}r_{t-1}+CBu_{t-1}-r_{t}\right\Vert_{2}^{2}\leq\xi_{t}^{2},~~ \xi_{t}>0\\
\Longleftrightarrow& ~\xi_{t}-\left(CAC^{\dag}r_{t-1}+CBu_{t-1}-r_{t}\right)^{\mathrm{T}}(\xi_{t}I)^{-1}\left(CAC^{\dag}r_{t-1}+CBu_{t-1}-r_{t}\right)\geq0,~~ \xi_{t}>0\\
\Longleftrightarrow& ~\xi_{t}-\left(CAC^{\dag}r_{t-1}+CBu_{t-1}-r_{t}\right)^{\mathrm{T}}(\xi_{t}I)^{-1}\left(CAC^{\dag}r_{t-1}+CBu_{t-1}-r_{t}\right)\geq0,~~ \xi_{t}I\in\textbf{S}_{++}^{p}\\
\Longleftrightarrow& ~\begin{pmatrix}
\xi_{t}I & CAC^{\dag}r_{t-1}+CBu_{t-1}-r_{t}\\
\left(CAC^{\dag}r_{t-1}+CBu_{t-1}-r_{t}\right)^{\mathrm{T}} & \xi_{t}
\end{pmatrix}\in\textbf{S}_{+}^{p+1},
\end{align*}
where the second equivalency follows from the fact that a symmetric matrix is positive definite if and only if all of its eigenvalues are positive \cite{Barrett2014}, and the third equivalency follows from Lemma \ref{lemma:3}. Therefore, we obtain the equivalent reformulation.
According to the SD representability assumptions and Lemma \ref{lemma:1}, (\ref{AMOPUL}) becomes an SDP problem.
\end{proof}

The computational tractability of SDP problems \cite{BenTal2001, Nesterov1994, Terlaky1996} and Theorem \ref{theorem:1} assure that (\ref{AMOPUL}) is polynomial-time solvable. Once an optimal solution $(A^{\rm a*}, U^{\rm a*}, \omega^{\rm a*})$ of (\ref{AMOPUL}) is obtained, it is worth estimating the cumulative error between the system outputs and the references in (\ref{MOPUL}). An upper bound for this gap is provided in the next theorem.
\begin{thm}\label{theorem:2}
If there exist $\beta>0$ and $\omega^{\rm u}>0$ such that
$\Vert CA^{\rm a}C^{\dag}\Vert_{2}\leq\beta$ and $\omega^{\rm a}\leq\omega^{\rm u}$ for any feasible $A^{\rm a}, \omega^{\rm a}$ of (\ref{AMOPUL}), then for each optimal solution $(A^{\rm a*}, U^{\rm a*}, \omega^{\rm a*})$ of (\ref{AMOPUL}), by letting $x_{t}=A^{\rm a*}x_{t-1}+Bu_{t-1}^{\rm a*}, y_{t}=Cx_{t}, t=1,2, \ldots, N$, we have
$$\sum_{t=1}^{N}\Vert y_{t}-r_{t}\Vert_{2}\leq \left(\sum_{i=0}^{N-1}\beta^{i}\right)\omega^{\rm u}.$$
\end{thm}
\begin{proof}
When $C$ has full column rank, we have $x_{0}=C^{\dag}r_{0}$ and $C^{\dag}C=I$ \cite{Penrose1956}. Hence
\begin{align*}
&\sum_{t=1}^{N}\left\Vert y_{t}-r_{t}\right\Vert_{2}\\
=&\sum_{t=1}^{N}\left\Vert C(A^{\rm a*}x_{t-1}+Bu_{t-1}^{\rm a*})-CA^{\rm a*}C^{\dag}r_{t-1}+CA^{\rm a*}C^{\dag}r_{t-1}-r_{t}\right\Vert_{2}\\
\leq&\sum_{t=1}^{N}\left\Vert CA^{\rm a*}C^{\dag}\left(Cx_{t-1}-r_{t-1}\right)\right\Vert_{2}+\sum_{t=1}^{N}\left\Vert y_{t}^{\rm a}-r_{t}\right\Vert_{2}\\
\leq&\sum_{t=1}^{N}\left\Vert CA^{\rm a*}C^{\dag}\left(y_{t-1}-r_{t-1}\right)\right\Vert_{2}+\omega^{\rm a*}\\
\leq&\left\Vert CA^{\rm a*}C^{\dag}\right\Vert_{2}\left(\sum_{t=1}^{N-1}\left\Vert y_{t}-r_{t}\right\Vert_{2}\right)+\omega^{\rm a*}.
\end{align*}
With the same arguments, we have
\begin{align*}
\sum_{t=1}^{n}\left\Vert y_{t}-r_{t}\right\Vert_{2}\leq\left\Vert CA^{\rm a*}C^{\dag}\right\Vert_{2}\left(\sum_{t=1}^{n-1}\left\Vert y_{t}-r_{t}\right\Vert_{2}\right)+\omega^{\rm a*},~~n=2, 3, \ldots, N-1.
\end{align*}
Consequently, we have
\begin{align*}
&\sum_{t=1}^{N}\left\Vert y_{t}-r_{t}\right\Vert_{2}\\
\leq&\left\Vert CA^{\rm a*}C^{\dag}\right\Vert_{2}^{N-1}\left\Vert y_{1}-r_{1}\right\Vert_{2}+\left(\sum\limits_{i=0}^{N-2}\left\Vert CA^{\rm a*}C^{\dag}\right\Vert_{2}^{i}\right)\omega^{\rm a*}\\
=&\left\Vert CA^{\rm a*}C^{\dag}\right\Vert_{2}^{N-1}\left\Vert y_{1}^{\rm a}-r_{1}\right\Vert_{2}+\left(\sum\limits_{i=0}^{N-2}\left\Vert CA^{\rm a*}C^{\dag}\right\Vert_{2}^{i}\right)\omega^{\rm a*}\\
\leq&\left(\sum\limits_{i=0}^{N-1}\left\Vert CA^{\rm a*}C^{\dag}\right\Vert_{2}^{i}\right)\omega^{\rm a*}\\
\leq&\left(\sum\limits_{i=0}^{N-1}\beta^{i}\right)\omega^{\rm u}.
\end{align*}
\end{proof}
\begin{rem}\label{remark:2}
For the application problems in Section \ref{section:2}, we may assume that the variables $A$ and $\omega$ are bounded. Then the assumptions $\Vert CA^{\rm a}C^{\dag}\Vert_{2}\leq\beta$ and $\omega^{\rm a}\leq\omega^{\rm u}$ for any feasible $A^{\rm a}$ and $\omega^{\rm a}$ of (\ref{AMOPUL}) in Theorem \ref{theorem:2} are satisfied. In this case, additional constraints such as $\Vert A\Vert_{2}\leq\alpha$ and $\omega\leq\omega^{\rm u}$ with large enough $\alpha$ and $\omega^{\rm u}$ can be added to $\mathcal{S}$ if necessary.
\end{rem}
When the control level $\omega$ in (\ref{MOPUL}) is fixed as a constant, we see the feasibility of an (\ref{AMOPUL}) solution to (\ref{MOPUL}) as below.
\begin{thm}\label{theorem:3}
Suppose that $\mathcal{S}\subseteq\{(A, U, \omega)|\omega=\omega^{\rm c}\}$ in (\ref{MOPUL}) with $\omega^{c}>0$ being given. Replace the constraint $\sum_{t=1}^{N}\Vert y_{t}^{\rm a}-r_{t}\Vert_{2}\leq \omega^{\rm c}$ with $\sum_{t=1}^{N}\Vert y_{t}^{\rm a}-r_{t}\Vert_{2}\leq\omega^{\rm c}/(\sum_{i=0}^{N-1}\beta^{i})$ in the SDP approximation (\ref{AMOPUL}) for some $\beta>0$.
If $\Vert CA^{\rm a}C^{\dag}\Vert_{2}\leq\beta$ for any feasible $A^{\rm a}$ of (\ref{AMOPUL}), then any feasible solution $(A^{\rm a}, U^{\rm a})$ of (\ref{AMOPUL}) is feasible to (\ref{MOPUL}), and the optimal objective value of (\ref{AMOPUL}) becomes an upper bound for that of (\ref{MOPUL}).
\end{thm}
\begin{proof}
With similar arguments as in the proof of Theorem \ref{theorem:2}, we have
\begin{align*}
\sum_{t=1}^{n}\left\Vert y_{t}-r_{t}\right\Vert_{2}\leq\left\Vert CA^{\rm a}C^{\dag}\right\Vert_{2}\left(\sum_{t=1}^{n-1}\left\Vert y_{t}-r_{t}\right\Vert_{2}\right)+\frac{\omega^{\rm c}}{\sum_{i=0}^{N-1}\beta^{i}},~~n=2, 3, \ldots, N,
\end{align*}
for any feasible solution $(A^{\rm a}, U^{\rm a})$ of (AMOPUL).
Hence
\begin{align*}
\sum_{t=1}^{N}\left\Vert y_{t}-r_{t}\right\Vert_{2}\leq&\left(\sum\limits_{i=0}^{N-1}\left\Vert CA^{\rm a}C^{\dag}\right\Vert_{2}^{i}\right)\frac{\omega^{\rm c}}{\sum_{i=0}^{N-1}\beta^{i}}\\
\leq&\left(\sum\limits_{i=0}^{N-1}\beta^{i}\right)\frac{\omega^{\rm c}}{\sum_{i=0}^{N-1}\beta^{i}}\\
=&~~\omega^{\rm c},
\end{align*}
which implies the feasibility of $(A^{\rm a}, U^{\rm a})$ to (\ref{MOPUL}). Then the optimal objective value of (\ref{AMOPUL}) becomes an upper bound for that of (\ref{MOPUL}).
\end{proof}
\noindent The above theorem shows that the control level plays an important role in (\ref{AMOPUL}). Numerically, we will study this issue further in Section \ref{section:4}.
\begin{rem}
For Theorem \ref{theorem:3}, if $\mathcal{S}\subseteq\{(A, U, \omega)|\Vert A\Vert_2\leq\alpha\}$, by taking $\beta=\alpha\Vert C\Vert_2\Vert C^{\dag}\Vert_2$, the assumption $\Vert CA^{\rm a}C^{\dag}\Vert_{2}\leq\beta$ is satisfied.
\end{rem}
\begin{rem}\label{remark:3}
A weighted cumulative error $\sum_{t=1}^{N}\Vert y_{t}-r_{t}\Vert_{Q}$ can also be used in (\ref{MOPUL}) by replacing $\Vert\cdot\Vert_{2}$ with $\Vert\cdot\Vert_{Q}$. The corresponding approximation model can then be constructed by using $\Vert\cdot\Vert_{Q}$ in (\ref{AMOPUL}). With similar arguments as in the proof of Theorem \ref{theorem:1}, its approximation is an SDP problem. As $\eta_{1}\Vert x\Vert_{Q}\leq\Vert x\Vert_{2}\leq\eta_{2}\Vert x\Vert_{Q}$ holds for all $x\in\mathbb{R}^{p}$ and some $\eta_{1}, \eta_{2}>0$ \cite[1.12, Proposition 4]{Zeidler1995}, Theorem \ref{theorem:3} follows when an upper bound of $\sum_{t=1}^{N}\Vert y_{t}^{\rm a}-r_{t}\Vert_{Q}$ is given by
\begin{align*}
\frac{\omega^{\rm c}}{(\frac{\eta_{2}\beta}{\eta_{1}})^{N-1}+\sum_{i=0}^{N-2}\frac{\eta_{2}^{i+1}\beta^{i}}{\eta_{1}^{i+1}}}.
\end{align*}
\end{rem}
\begin{rem}\label{remark:4}
When the matrix $A$ is time-varying in (\ref{MOPUL}), i.e.,  $x_{t}=A_{t-1}x_{t-1}+Bu_{t-1}$, we can also get a similar SDP approximation with similar discussions.
\end{rem}

\subsection{Two special cases}\label{subsection:3.3}
We study two special cases of (\ref{MOPUL}) associated with specific application problems in Section \ref{section:2} focusing on the reference outputs fitting and transition matrix estimation, respectively.
\subsubsection{MOPUL1}\label{subsubsection:3.3.1}
To fit the given reference outputs $\{r_t\}_{t=1}^{N}$, we consider the following (\ref{MOPUL1}) problem to minimize the cumulative error of reference outputs:
\begin{align}
\begin{split}\label{MOPUL1}
\min_{A, U}&~~\sum_{t=1}^{N}\left\Vert y_{t}-r_{t}\right\Vert_{2}\\
{\rm s.t.}
&~~x_{t}=Ax_{t-1}+Bu_{t-1},~~t=1, 2,\ldots, N,\\
&~~y_{t}=Cx_{t},~~t=0, 1, \ldots, N,\\
&~~\left(A, U\right)\in\mathcal{S}_{\rm\scriptscriptstyle M1},
\end{split}\tag{MOPUL1}
\end{align}
where the transition matrix $A\in\mathbb{R}^{n\times n}$ and control $U\coloneqq(u_{0}, u_{1}, \ldots, u_{N-1})\in\mathbb{R}^{m\times N}$ are decision variables, $\mathcal{S}_{\rm\scriptscriptstyle M1}$ is SD representable. It is a special case of (\ref{MOPUL}) by setting $\lambda_{1}=\lambda_{2}=0$, $\lambda_{3}=1$, $f_{3}(\omega)=\omega$, and $\mathcal{S}=\mathcal{S}_{\rm\scriptscriptstyle M1}\times\mathbb{R}_+$. Notice that (\ref{O-MPC}) with $\lambda=0$ in Subsection \ref{subsection:2.1},
(\ref{O-COVID}) in Subsection \ref{subsection:2.2}, (\ref{O-Markov}) in Subsection \ref{subsection:2.3}, and (\ref{O-IN/OUTPUT1}) in Subsection \ref{subsection:2.4} are four examples of (\ref{MOPUL1}). An SDP approximation model of (\ref{MOPUL1}) becomes the following:
\begin{align}
\begin{split}\label{AMOPUL1}
\min_{A, U}&~~\sum_{t=1}^{N}\left\Vert y_{t}^{\rm a}-r_{t}\right\Vert_{2}\nonumber\\
{\rm s.t.}
&~~r_{0}=Cx_{0},~~x_{0}^{\rm a}=x_{0},\\
&~~x_{t}^{\rm a}=AC^{\dag}r_{t-1}+Bu_{t-1},~~t=1, 2,\ldots, N,\\
&~~y_{t}^{\rm a}=Cx_{t}^{\rm a},~~t=0, 1, \ldots, N,\\
&~~\left(A, U\right)\in\mathcal{S}_{\rm\scriptscriptstyle M1},
\end{split}\tag{AMOPUL1}
\end{align}
where $A$ and $U\coloneqq(u_{0}, u_{1}, \ldots, u_{N-1})$ are decision variables.

The next theorem provides an upper bound for the optimal objective value $v_{\rm\scriptscriptstyle A1}^*$ of (\ref{AMOPUL1}) assuming that the given reference outputs are accurate.
\begin{thm}\label{theorem:4}
Let $\{\epsilon_{t}\}_{t=1}^N$ be $N$ nonnegative constants. For problem (\ref{AMOPUL1}), if the set
\begin{align}\label{Z_epsilon}
\mathcal{Z}_\epsilon\coloneqq\left\{(A, U)\in\mathcal{S}_{\rm\scriptscriptstyle M1}\left|\begin{array}{l}U=(u_0, u_1,\dots, u_{N-1}),\\
\hat{y}_{0}=Cx_{0},\\
\hat{y}_{t}=CAC^{\dag}\hat{y}_{t-1}+CBu_{t-1},\\
\Vert \hat{y}_t-r_{t}\Vert_{2}\leq\epsilon_{t},~~t=1, 2, \ldots, N\end{array}\right.\right\}\neq\emptyset,
\end{align}
then we have
\begin{align*}
v_{\rm\scriptscriptstyle A1}^*\leq\left(1+\gamma\right)\sum_{t=1}^{N-1}\epsilon_{t}+\epsilon_{N},
\end{align*}
where $\gamma\coloneqq\inf_{\scriptscriptstyle(A, U)\in \mathcal{Z}_\epsilon}\Vert CAC^{\dag}\Vert_2$.
Moreover, if $\{r_{t}\}_{t=1}^N$ are accurate system references, i.e., there exists an $(A, U)\in\mathcal{S}_{\rm\scriptscriptstyle M1}$ such that $\hat{y}_{t}=r_{t}, t=1, 2, \ldots, N$, then $v_{\rm\scriptscriptstyle A1}^*=0$.
\end{thm}
\begin{proof}
For each approximate iteration $y_{t}^{\rm a}=CAC^{\dag}r_{t-1}+CBu_{t-1}, t=1, 2, \ldots, N$ with $(A, U)\in\mathcal{Z}_\epsilon$,
\begin{align*}
&\sum\limits_{t=1}^{N}\left\Vert y_{t}^{\rm a}-r_{t}\right\Vert_{2}\\
=&\sum\limits_{t=1}^{N}\left\Vert CAC^{\dag}r_{t-1}+CBu_{t-1}-r_{t}\right\Vert_{2}\\
=&\sum\limits_{t=1}^{N}\left\Vert CAC^{\dag}(r_{t-1}-\hat{y}_{t-1}+\hat{y}_{t-1})+CBu_{t-1}-(r_{t}-\hat{y}_{t}+\hat{y}_{t})\right\Vert_{2}\\
\leq&\sum_{t=1}^{N}\left\Vert CAC^{\dag}\left(\hat{y}_{t-1}-r_{t-1}\right)\right\Vert_{2}+\sum_{t=1}^{N}\left\Vert \hat{y}_{t}-r_{t}\right\Vert_{2}\\
\leq&\left\Vert CAC^{\dag}\right\Vert_{2}\sum_{t=1}^{N-1}\left\Vert \hat{y}_{t}-r_{t}\right\Vert_{2}+\sum_{t=1}^{N}\left\Vert \hat{y}_{t}-r_{t}\right\Vert_{2}\\
\leq&\left(1+\Vert CAC^{\dag}\Vert_{2}\right)\sum_{t=1}^{N-1}\epsilon_{t}+\epsilon_{N}.
\end{align*}
Thus
\begin{align*}
v_{\rm\scriptscriptstyle A1}^*&\leq\sum\limits_{t=1}^{N}\left\Vert y_{t}^{\rm a}-r_{t}\right\Vert_{2}\leq\left(1+\Vert CAC^{\dag}\Vert_{2}\right)\sum_{t=1}^{N-1}\epsilon_{t}+\epsilon_{N}
\end{align*}
holds for each $(A, U)\in\mathcal{Z}_\epsilon$.
Consequently, we have
\begin{align*}
v_{\rm\scriptscriptstyle A1}^*&\leq\inf_{\scriptscriptstyle(A, U)\in \mathcal{Z}_\epsilon}\left\{\left(1+\Vert CAC^{\dag}\Vert_{2}\right)\sum_{t=1}^{N-1}\epsilon_{t}+\epsilon_{N}\right\}=\left(1+\gamma\right)\sum_{t=1}^{N-1}\epsilon_{t}+\epsilon_{N},
\end{align*}
and the rest of theorem follows.
\end{proof}
This theorem shows that the accuracy of reference outputs is important to (\ref{AMOPUL1}).
Given a sequence of accurate reference outputs, (\ref{AMOPUL1}) can fit them without error. Numerically, we will study this issue further in Section \ref{section:4}.

Let $v_{\rm\scriptscriptstyle M1}^*$ and $v_{\rm\scriptscriptstyle A1}^*$ be the optimal objective values of (\ref{MOPUL1}) and (\ref{AMOPUL1}), respectively, some bounds on $v_{\rm\scriptscriptstyle A1}^*/v_{\rm \scriptscriptstyle M1}^*$ are provided in the next two results.
\begin{thm}\label{theorem:5}
If problem (\ref{MOPUL1}) is attainable, then we have
\begin{align*}
v_{\rm\scriptscriptstyle A1}^*\leq\left(1+\gamma^*\right)v_{\rm\scriptscriptstyle M1}^*,
\end{align*}
where $\gamma^*\coloneqq\inf_{\scriptscriptstyle(A^*, U^*)\in\mathcal{F}_{\rm\scriptscriptstyle M1}^{*}}\inf_{\scriptscriptstyle(A, U)\in \mathcal{Z}_{\epsilon}^*}\Vert CAC^{\dag}\Vert_2$ with $\mathcal{F}_{\rm\scriptscriptstyle M1}^{*}$ being the optimal solution set of (\ref{MOPUL1}) and $\mathcal{Z}_{\epsilon}^*$ being the set defined in (\ref{Z_epsilon}) in which $\epsilon_t=\Vert \hat{y}_{t}^*-r_{t}\Vert_{2}$, $\hat{y}_{t}^*=CA^*C^{\dag}\hat{y}_{t-1}^*+CBu^*_{t-1}, t=1, 2, \ldots, N$, and $\hat{y}_{0}^*=Cx_{0}$ for each $(A^*, U^*)\in\mathcal{F}_{\rm\scriptscriptstyle M1}^{*}$.
\end{thm}
\begin{proof}
When (\ref{MOPUL1}) is attainable, we have $(A^*,U^*)\in \mathcal{Z}_{\scriptscriptstyle \epsilon^*}\ne\emptyset$ for each $(A^*,U^*)\in\mathcal{F}_{\rm\scriptscriptstyle M1}^{*}$. Theorem \ref{theorem:4} says that
\begin{align*}
v_{\rm\scriptscriptstyle A1}^*\leq\left(1+\inf_{\scriptscriptstyle(A, U)\in \mathcal{Z}_{\epsilon^{*}}}\Vert CAC^{\dag}\Vert_2\right)v_{\rm\scriptscriptstyle M1}^*
\end{align*}
holds for each $(A^*,U^*)\in\mathcal{F}_{\rm\scriptscriptstyle M1}^{*}$. Hence
\begin{align*}
v_{\rm\scriptscriptstyle A1}^*&\leq\inf_{\scriptscriptstyle(A^*, U^*)\in\mathcal{F}_{\rm\scriptscriptstyle M1}^{*}}\left(1+\inf_{\scriptscriptstyle(A, U)\in \mathcal{Z}_{\epsilon^{*}}}\Vert CAC^{\dag}\Vert_2\right)v_{\rm\scriptscriptstyle M1}^*=\left(1+\gamma^*\right)v_{\rm\scriptscriptstyle M1}^*.
\end{align*}
\end{proof}
\begin{thm}\label{theorem:6}
If problem (\ref{AMOPUL1}) is attainable, then we have
\begin{align*}
v_{\rm\scriptscriptstyle M1}^*\leq \left(\sum_{i=0}^{N-1}(\zeta^*)^i\right) v_{\rm\scriptscriptstyle A1}^*,
\end{align*}
where $\zeta^*\coloneqq\inf_{\scriptscriptstyle(A^{\rm a*},U^{\rm a*})\in \mathcal{F}_{\rm\scriptscriptstyle A1}^{*}}\Vert CA^{\rm a*}C^{\dag}\Vert_{2}$ with $\mathcal{F}_{\rm\scriptscriptstyle A1}^{*}$ being the optimal solution set of (\ref{AMOPUL1}).
\end{thm}
\begin{proof}
When (\ref{AMOPUL1}) is attainable, let $(A^{\rm a*}, U^{\rm a*})\in\mathcal{F}_{\rm\scriptscriptstyle A1}^{*}$ be an optimal solution of\\ (\ref{AMOPUL1}), $x_{t}=A^{\rm a*}x_{t-1}+Bu_{t-1}^{\rm a*}$, $y_{t}=Cx_{t}$ in (\ref{MOPUL1}), and $x^{\rm a*}_{t}=A^{\rm a*}C^{\dag}r_{t-1}+Bu_{t-1}^{\rm a*}$, $y^{\rm a*}_{t}=Cx^{\rm a*}_{t}$ in (\ref{AMOPUL1}), $t=1, 2, \ldots, N$. With similar arguments as in the proof of Theorem \ref{theorem:2}, we can obtain the following $N-1$ inequalities:
\begin{align*}
\sum_{t=1}^{n}\Vert y_{t}-r_{t}\Vert_{2}\leq\Vert CA^{\rm a*}C^{\dag}\Vert_{2}\sum_{t=1}^{n-1}\Vert y_{t}-r_{t}\Vert_{2}+\sum_{t=1}^{N}\left\Vert y_{t}^{\rm a*}-r_{t}\right\Vert_{2},~~n=2, 3, \ldots, N,
\end{align*}
which imply that
\begin{align*}
\sum_{t=1}^{N}\Vert y_{t}-r_{t}\Vert_{2}&\leq\left(\sum_{i=0}^{N-1}\Vert CA^{\rm a*}C^{\dag}\Vert_{2}^{i}\right)\sum_{t=1}^{N}\left\Vert y_{t}^{\rm a*}-r_{t}\right\Vert_{2}\\
&=\left(\sum_{i=0}^{N-1}\Vert CA^{\rm a*}C^{\dag}\Vert_{2}^{i}\right)v_{\rm\scriptscriptstyle A1}^*.
\end{align*}
Since any optimal solution of (\ref{AMOPUL1}) is feasible to (\ref{MOPUL1}), we know
\begin{align*}
v_{\rm\scriptscriptstyle M1}^*\leq \left(\sum_{i=0}^{N-1}\Vert CA^{\rm a*}C^{\dag}\Vert_{2}^{i}\right)v_{\rm\scriptscriptstyle A1}^*
\end{align*}
holds for each $(A^{\rm a*}, U^{\rm a*})\in\mathcal{F}_{\rm\scriptscriptstyle A1}^{*}$. Hence
\begin{align*}
v_{\rm\scriptscriptstyle M1}^*&\leq \inf_{\scriptscriptstyle(A^{\rm a*},U^{\rm a*})\in \mathcal{F}_{\rm\scriptscriptstyle A1}^{*}}\left(\sum_{i=0}^{N-1}\Vert CA^{\rm a*}C^{\dag}\Vert_{2}^{i}\right)v_{\rm\scriptscriptstyle A1}^*\\
&=\left(\sum_{i=0}^{N-1}(\zeta^*)^{i}\right)v_{\rm\scriptscriptstyle A1}^*.
\end{align*}
\end{proof}
\begin{rem}\label{remark:5}
When (\ref{MOPUL1}) and (\ref{AMOPUL1}) are both attainable with $v_{\rm \scriptscriptstyle M1}^*>0$, Theorems \ref{theorem:5} and \ref{theorem:6} imply that
\begin{align*}
\frac{1}{\sum_{i=0}^{N-1}(\zeta^*)^i}\leq \frac{v_{\rm\scriptscriptstyle A1}^*}{v_{\rm \scriptscriptstyle M1}^*}\leq 1+\gamma^*.
\end{align*}
Furthermore, if $\zeta^*=\gamma^*=0$, then $v_{\rm\scriptscriptstyle M1}^*=v_{\rm \scriptscriptstyle A1}^*$.
\end{rem}
\subsubsection{MOPUL2}\label{subsubsection:3.3.2}
In some scenarios such as (\ref{O-IN/OUTPUT2}) in Subsection \ref{subsection:2.4} and a COVID-19 pandemic optimal control model MOCM in \cite{Xu2023}, a small change of the transition matrix $A$ is preferred within a guaranteed level of cumulative error. We may consider the following problem:
\begin{align}
\begin{split}\label{MOPUL2}
\min_{A, U}&~~\left\Vert A-A^{\rm r}\right\Vert_{F}\nonumber\\
{\rm s.t.}
&~~x_{t}=Ax_{t-1}+Bu_{t-1},~~t=1, 2, \ldots, N,\\
&~~y_{t}=Cx_{t},~~t=0, 1, \ldots, N,\\
&~~\sum_{t=1}^{N}\left\Vert y_{t}-r_{t}\right\Vert_{2}\leq \omega,\\
&~~\left\Vert u_{t}-u_{t}^{\rm r}\right\Vert_{2}\leq \omega_{t},~~t=0, 1, \ldots, N-1,\\
&~~\left(A, U\right)\in\mathcal{S}_{\rm\scriptscriptstyle M2},
\end{split}\tag{MOPUL2}
\end{align}
where the transition matrix $A\in\mathbb{R}^{n\times n}$ and control $U\coloneqq(u_{0}, u_{1}, \ldots, u_{N-1})\in\mathbb{R}^{m\times N}$ are decision variables, $\{r_t\}_{t=1}^{N}\subseteq\mathbb{R}^{p}$, $A^{\rm r}\in\mathbb{R}^{n\times n}$, and $\{u_{t}^{\rm r}\}_{t=0}^{N-1}\subseteq\mathbb{R}^{m\times 1}$ are given references of $\{y_t\}_{t=1}^{N}$, $A$, and $\{u_{t}\}_{t=0}^{N-1}$, respectively, $\omega, \{\omega_{t}\}_{t=0}^{N-1}\subseteq\mathbb{R}_{+}$ are control levels, and constraint set $\mathcal{S}_{\rm\scriptscriptstyle M2}\subseteq\mathbb{R}^{n\times n}\times\mathbb{R}^{m\times N}$ is SD representable. Notice that this is a special case of (\ref{MOPUL}) by setting $\lambda_{1}=1$, $\lambda_{2}=\lambda_{3}=0$, $f_{1}(A)=\Vert A-A^{\rm r}\Vert_{F}$, and $\mathcal{S}=(\mathbb{R}^{n\times n}\times\{U|\Vert u_{t}-u_{t}^{\rm r}\Vert_{2}\leq\omega_{t}, t=0, 1, \ldots, N-1\}\cap\mathcal{S}_{\rm\scriptscriptstyle M2})\times\{{\rm constant}~\omega\}$.
The objective function enforces a steady and controllable change of transition matrix $A$ with respect to the reference value. The third constraint guarantees a cumulative precision of the iteration within the control level $\omega$. And the fourth constraint guarantees a controllable change of the system inputs. Correspondingly, we can derive the following SDP approximation model:
\begin{align}
\begin{split}\label{AMOPUL2}
\min_{A, U}&~~\left\Vert A-A^{\rm r}\right\Vert_{F}\\
{\rm s.t.}
&~~r_{0}=Cx_{0},~~x_{0}^{\rm a}=x_{0},\\
&~~x_{t}^{\rm a}=AC^{\dag}r_{t-1}+Bu_{t-1},~~t=1, 2, \ldots, N,\\
&~~y_{t}^{\rm a}=Cx_{t}^{\rm a},~~t=0, 1, \ldots, N,\\
&~~\sum_{t=1}^{N}\left\Vert y_{t}^{\rm a}-r_{t}\right\Vert_{2}\leq\tilde{\omega},\\
&~~\left\Vert u_{t}-u_{t}^{\rm r}\right\Vert_{2}\leq \omega_{t},~~t=0, 1, \ldots, N-1,\\
&~~\left(A, U\right)\in\mathcal{S}_{\rm\scriptscriptstyle M2},
\end{split}\tag{AMOPUL2}
\end{align}
where $A$ and $U\coloneqq(u_{0}, u_{1}, \ldots, u_{N-1})$ are decision variables, $\tilde{\omega}\in\mathbb{R}_+$ is control level. Following Theorem \ref{theorem:3}, we have the relationship between (\ref{MOPUL2}) and (\ref{AMOPUL2}) in the next result.
\begin{thm}\label{theorem:7}
Take $\tilde{\omega}=\omega/(\sum_{i=0}^{N-1}\beta^{i})$ in (\ref{AMOPUL2}) with respect to $\omega$ in (\ref{MOPUL2}) and $\beta>0$.
If $\Vert CA^{\rm a}C^{\dag}\Vert_{2}\leq\beta$ for any feasible $A^{\rm a}$ of (\ref{AMOPUL2}), then any feasible solution of (\ref{AMOPUL2}) is feasible to (\ref{MOPUL2}), and the optimal objective value of (\ref{AMOPUL2}) becomes an upper bound for that of (\ref{MOPUL2}).
\end{thm}
\begin{rem}\label{remark:6}
The existence of an upper bound $\beta$ of $\Vert CA^{\rm a}C^{\dag}\Vert_{2}$ in Theorem \ref{theorem:7} is guaranteed as mentioned in Remark \ref{remark:2}.
\end{rem}

\section{Numerical experiments}\label{section:4}
In this section, we study the influences of perturbed noises at reference outputs $\{r_t\}_{t=1}^N$ and control level $\omega$ on the performance of the proposed approximation model (\ref{AMOPUL}) numerically.
Theorem \ref{theorem:4} shows that the noise levels of the given reference outputs $\{r_t\}_{t=1}^N$ are keys to the optimal objective value.
We study the numerical performance of (\ref{AMOPUL1}) and (\ref{AMOPUL2}) with different noise levels of $\{r_t\}_{t=1}^N$ in Subsection \ref{subsection:4.1} and Subsubsection \ref{subsubsection:4.2.1}, respectively.
In addition, Theorems \ref{theorem:3} and \ref{theorem:7} indicate that the size of feasible set of (\ref{AMOPUL}) is mainly determined by the control level $\tilde{\omega}$. Related numerical results on the performance of (\ref{AMOPUL2}) in terms of $\tilde{\omega}$ are reported in Subsubsection \ref{subsubsection:4.2.2}.

All data are randomly generated in our experiments as following:
\begin{itemize}[leftmargin=*]
  \item {\bf An ideal instance.} Take $n=m=p=100$, $N=30$, and $B=C=I$ in (\ref{MOPUL}).
Initial $r_{0}$ is uniformly generated in $(-0.5, 0.5)^p$. $\hat{A}$ is an ideal value of $A$ with each component being generated from the normal distribution $\mathcal{N}(0, 0.1^2)$. $\hat{u}_{t}\coloneqq\boldsymbol{1}\hat{u}_t^{\rm s}$ is an ideal value of $u_t$ with $\hat{u}_t^{\rm s}$ being uniformly generated in $(-0.5, 0.5)$, $t=0, 1, \ldots, N-1$.
Define $\hat{x}_{t}=\hat{A}\hat{x}_{t-1}+\hat{u}_{t-1}, t=1, 2, \dots, N$ and $\hat{x}_0=r_0$. Then $(\hat{A}, \{\hat{u}_t\}_{t=0}^{N-1}, \{\hat{x}_t\}_{t=0}^N)$ forms an ideal instance.
  \item {\bf Reference outputs.} Reference output $r_{t}\coloneqq\hat{x}_{t}+e_{t}$ with the ideal value $\hat{x}_{t}$ and a perturbed noise $e_{t}$ with each component being generated from $\mathcal{N}(\mu, \sigma^{2})$, $t=1, 2, \ldots, N$. A total of 20 random instances are generated for each $(\mu, \sigma)$.
\end{itemize}
The following 4 evaluation criteria are used to measure the performance of (\ref{AMOPUL1}) and (\ref{AMOPUL2}):
\begin{itemize}[leftmargin=*]
  \item[(1)] \textbf{Cumulative error (CE)}: $\sum_{t=1}^{N}\Vert x_t^*-r_{t}\Vert_{2}$ measures the cumulative precision of (\ref{AMOPUL1}), where $\{r_t\}_{t=0}^N$ are the generated reference outputs, $x_t^*=A^{\rm a*}x_{t-1}^*+u_{t-1}^{\rm a*}, t=1, 2, \ldots, N$, and $x_0^*=r_0$ for an optimal solution $(A^{\rm a*}, (u_0^{\rm a*}, u_1^{\rm a*}, \ldots, u_{N-1}^{\rm a*}))$ of (\ref{AMOPUL1}).
  \item[(2)] \textbf{Approximate cumulative error (ACE)}: $\sum_{t=1}^{N}\Vert x_t^{\rm a*}-r_{t}\Vert_{2}$ measures the approximate cumulative precision of (\ref{AMOPUL2}), where $\{r_t\}_{t=0}^N$ are the generated reference outputs, $x_t^{\rm a*}=A^{\rm a*}r_{t-1}+u_{t-1}^{\rm a*}$, $t=1, 2, \ldots, N$ for an optimal solution $(A^{\rm a*}, (u_0^{\rm a*}, u_1^{\rm a*}, \ldots, u_{N-1}^{\rm a*}))$ of (\ref{AMOPUL2}).
  \item[(3)] \textbf{Relative error of $A^{\rm a*}$ (RE$\boldsymbol{A^{\rm a*}}$)}: $\frac{\Vert A^{\rm a*}-\hat{A}\Vert_{F}}{\Vert \hat{A}\Vert_{F}}$ measures the difference between the true $\hat{A}$ and an approximate solution $A^{\rm a*}$.
  \item[(4)] \textbf{Relative error of $U^{\rm a*}$ (RE$\boldsymbol{U^{\rm a*}}$)}: $\frac{\Vert U^{\rm a*}-\hat{U}\Vert_{F}}{\Vert \hat{U}\Vert_{F}}$ measures the difference between the true $\hat{U}$ and an approximate solution $U^{\rm a*}$, where $\hat{U}\coloneqq(\hat{u}_{0}, \hat{u}_{1}, \ldots, \hat{u}_{N-1})$ and $U^{a*}\coloneqq(u_{0}^{a*}, u_{1}^{a*}, \ldots, u_{N-1}^{a*})$.
\end{itemize}
Numerical results using real data for the COVID-19 pandemic optimal control can be referred to \cite{Xu2023}.

All experiments are implemented using MATLAB R2019b on a laptop equipped with 4.00 GB memory and AMD Ryzen 3 2200U with Radeon Vega Mobile Gfx (2.50 GHz). We use MOSEK (version 9.1.9) (\href{https://www.mosek.com}{https://www.mosek.com}) in CVX-w64 (version 2.2) (\href{http://cvxr.com/cvx/}{http://cvxr.com/cvx/}) to solve all involved optimization problems. Five significant digits are taken for the numerical results shown in every table.

\subsection{Performance of AMOPUL1}\label{subsection:4.1}
We study the influences of noise levels at the reference outputs on the performance of (\ref{AMOPUL1}).
For simplicity, the constraint set $\mathcal{S}_{\rm\scriptscriptstyle M1}$ is set to be box constrained:
\begin{align*}
&\{A|-0.4\leq a_{ij}\leq0.4, i, j=1, 2, \ldots, n\}\times\{U|-0.5\leq u_{t}^{i}\leq0.5, i=1, 2, \ldots, n, t=0, 1, \ldots, N-1\},
\end{align*}
which is bounded and hence (\ref{AMOPUL1}) is attainable. Then (\ref{AMOPUL1}) can be reformulated as
\begin{align*}
\min_{A, U}&~~\sum_{t=1}^{N}\left\Vert Ar_{t-1}+u_{t-1}-r_{t}\right\Vert_{2}\\
{\rm s.t.}
&~~r_{0}=x_{0},\nonumber\\
&~~-0.4\leq a_{ij}\leq0.4,~~i, j=1, 2, \ldots, n,\\
&~~-0.5\leq u_{t}^{i}\leq0.5,~~i=1, 2, \ldots, n, t=0, 1, \ldots, N-1.\nonumber
\end{align*}
A total of 20 instances of $\{r_t\}_{t=1}^{N}$ with perturbed noises for each $(\mu,\sigma)$ are generated with respect to $\mu=0, \sigma=$ 0.05, 0.1, 0.2, 0.3, 0.4, 0.5, 0.6, 0.7, 0.8; and $\mu=1, \sigma=$ 2.5, 3, respectively. The means and standard deviations of the cumulative errors (CE) and relative errors of $A^{\rm a*}$ and $U^{\rm a*}$ (RE$A^{\rm a*}$ and RE$U^{\rm a*}$) are shown in Table \ref{table:1}, where $(A^{\rm a*}, U^{\rm a*})$ is the optimal solution of (\ref{AMOPUL1}).
\begin{table}[H]
\centering
\caption{Means and standard deviations of CE, RE$A^{\rm a*}$ and RE$U^{\rm a*}$ for (\ref{AMOPUL1}).}
\label{table:1}
\begin{tabular}{ccccccccc}
\hline
\multirow{3}*{($\mu$, $\sigma$)}  & \multicolumn{2}{c}{CE} &  \multicolumn{2}{c}{RE$A^{\rm a*}$} &  \multicolumn{2}{c}{ RE$U^{\rm a*}$} \\
\cmidrule(l){2-3} \cmidrule(l){4-5} \cmidrule(l){6-7}
& mean & std & mean & std & mean & std\\
\hline
(0, 0.05)    & 1.9885e-08 &  4.9712e-08 & 1.0070e+00 & 4.8692e-03 & 7.7927e-01 & 9.3722e-03\\
(0, 0.1)& 1.3192e-08 & 2.7314e-08& 1.0190e+00 & 7.4878e-03 & 8.1443e-01 & 1.5714e-02   \\
(0, 0.2) & 1.8346e-08& 2.5523e-08 & 1.0324e+00 & 8.0172e-03 & 8.8606e-01 & 1.2572e-02  \\
(0, 0.3)  & 4.1105e-08 & 1.2360e-07 & 1.0485e+00 & 9.6251e-03 & 9.2916e-01 & 1.0614e-02   \\
(0, 0.4)  & 7.6893e-09 & 2.0765e-08 & 1.0725e+00 & 1.1846e-02 & 9.5445e-01 & 1.1122e-02   \\
(0, 0.5)  & 1.2974e-08 & 3.5385e-08 &  1.0949e+00 & 9.8607e-03 & 9.6895e-01 & 7.6003e-03   \\
(0, 0.6)   & 3.6863e-08 & 5.0465e-08 & 1.1194e+00 & 1.7583e-02 & 9.7755e-01 & 7.0614e-03   \\
(0, 0.7)  & 7.2976e-09 & 1.7096e-08 & 1.1439e+00 & 1.0911e-02 & 9.8106e-01 & 3.4137e-03 \\
(0, 0.8)  & 1.5437e-08 & 2.9897e-08 & 1.1655e+00 & 2.1094e-02 & 9.8639e-01 & 3.6822e-03  \\
\textbf{(1, 2.5)}  & \textbf{1.6904e+01} & \textbf{4.3973e+01} & \textbf{1.6462e+00} & \textbf{6.3211e-02} & \textbf{1.0176e+00} & \textbf{1.2296e-02}  \\
\textbf{(1, 3.0)} & \textbf{8.6961e+01} & \textbf{9.3299e+01} & \textbf{1.8027e+00} & \textbf{7.9844e-02} & \textbf{1.0479e+00} & \textbf{2.4774e-02} \\
\hline
\end{tabular}
\end{table}
Figures \ref{figure:1} and \ref{figure:2} plot the trends of CE, RE$A^{\rm a*}$ and RE$U^{\rm a*}$ shown in Table \ref{table:1} with respect to the perturbed $(\mu, \sigma)$ pairs, respectively, in which the length of each error bar above and below the mean value reflects the corresponding standard deviation.
\begin{figure}[H]
\begin{subfigure}[\rm Line plot with error bars of CE.]{
\centering
\includegraphics[width=0.45\textwidth]{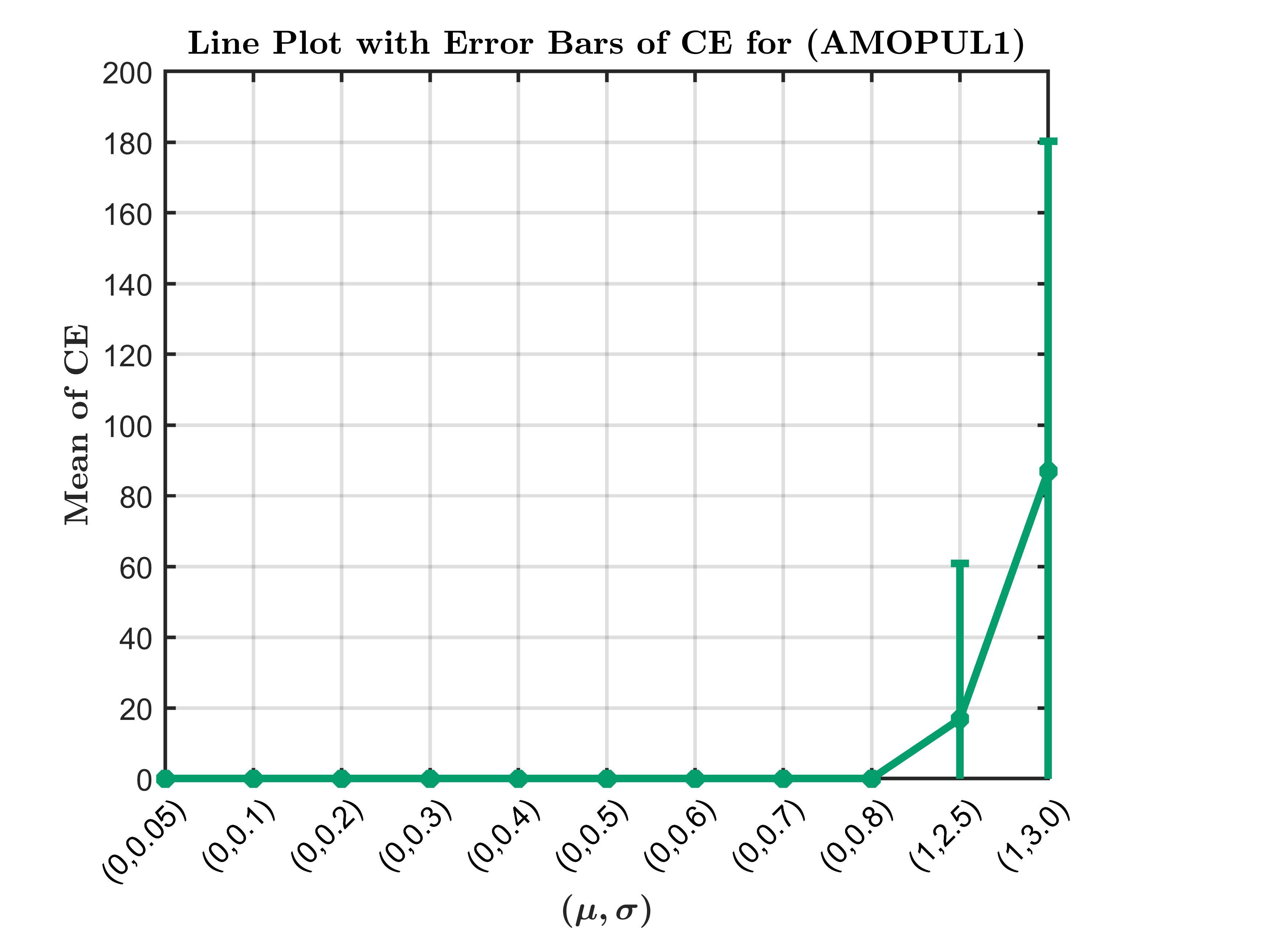}
\label{figure:1}}
\end{subfigure}
\begin{subfigure}[\rm Line plot with error bars of RE$A^{\rm a*}$ and RE$U^{\rm a*}$.]{
\centering
\includegraphics[width=0.45\textwidth]{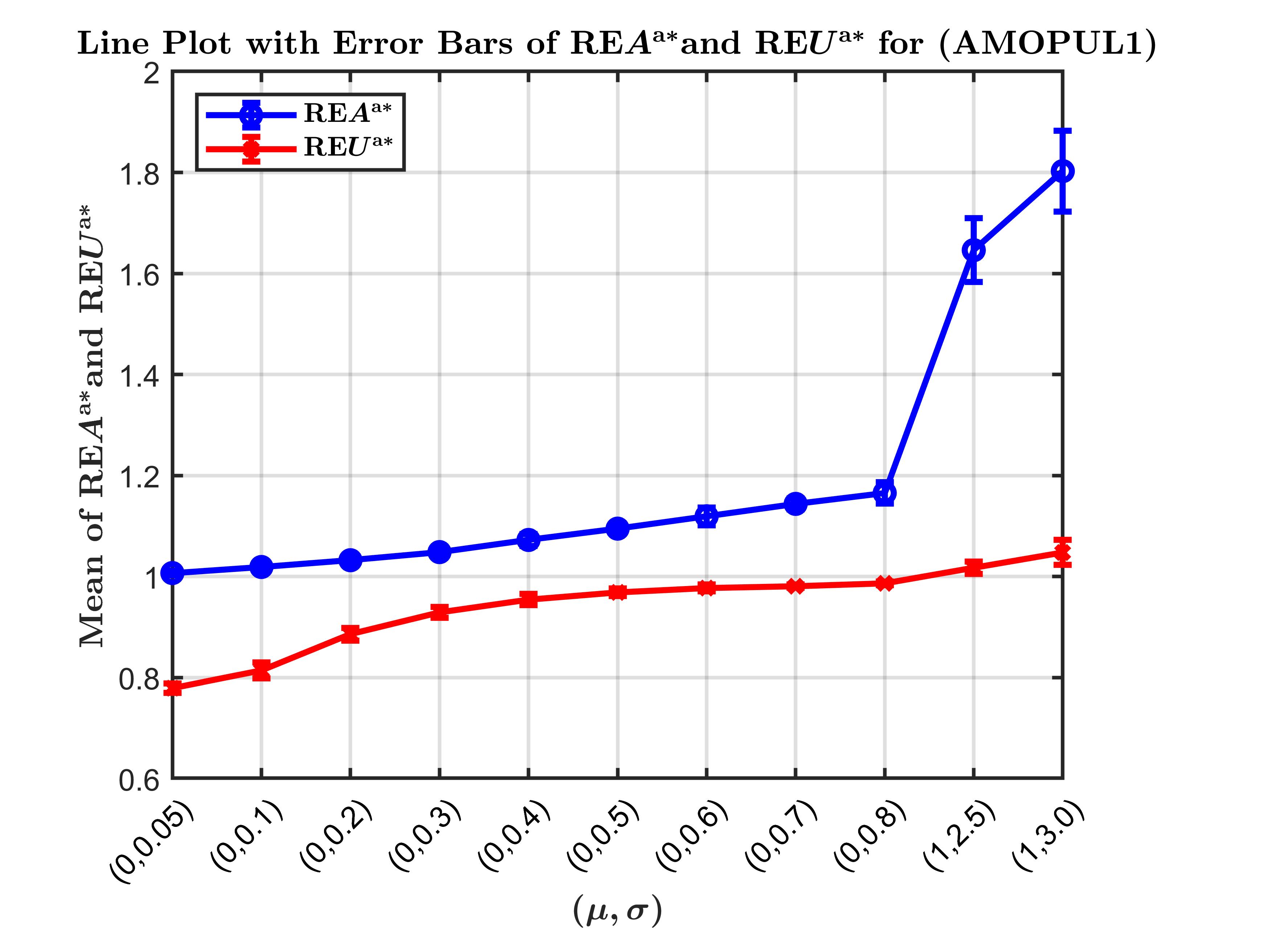}
\label{figure:2}}
\end{subfigure}
\caption{Line plots with error bars of CE, RE$A^{\rm a*}$ and RE$U^{\rm a*}$ for (\ref{AMOPUL1}).}
\end{figure}
Observe that when the perturbed noise of $\{r_t\}_{t=1}^{N}$ is small (say, $\mu=0, \sigma=0.05$ - $0.8$), the means and standard deviations of CE are all less than $10^{-6}$ as shown in Table \ref{table:1} and Figure \ref{figure:1}. This shows that (\ref{AMOPUL1}) achieves accurate and robust solutions.
When the perturbed noise becomes large (say, $\mu=1, \sigma=2.5$, $3$), the reference outputs $\{r_t\}_{t=1}^N$ become chaotic. In this case, the approximate output values fail to fit the given reference outputs, and CE becomes
large and oscillating.

On the other hand, RE$A^{\rm a*}$ and RE$U^{\rm a*}$ are more sensitive to the perturbed noises.
The means of RE$A^{\rm a*}$ and RE$U^{\rm a*}$ in Figure \ref{figure:2}
increase as the perturbed noise becomes larger.
In particular, when the noise is large enough (say, $\mu=1, \sigma=2.5$, $3$), there is a significant increase in the mean and standard deviation of RE$A^{\rm a*}$.

In summary, (\ref{AMOPUL1}) handles the system output quite well in fitting given reference outputs with small perturbed noises.

\subsection{Performance of AMOPUL2}\label{subsection:4.2}
We now study the performance of (\ref{AMOPUL2}) with perturbed noises at reference outputs $\{r_t\}_{t=1}^N$ and different control levels $\tilde{\omega}$ and $\{\omega_t\}_{t=0}^{N-1}$.
The constraint set $\mathcal{S}_{\rm\scriptscriptstyle M2}$ is set to be the trivial constraint $\mathbb{R}^{n\times n}\times\mathbb{R}^{n\times N}$ for simplicity. Then (\ref{AMOPUL2}) can be reformulated as
\begin{align}
\begin{split}\label{amopul2_special}
\min_{A, U}&~~\Vert A-\hat{A}\Vert_{F}\\
{\rm s.t.}
&~~r_{0}=x_{0},\\
&~~\sum_{t=1}^{N}\left\Vert Ar_{t-1}+u_{t-1}-r_{t}\right\Vert_{2}\leq \tilde{\omega},\\
&~~\left\Vert u_{t}-\hat{u}_{t}\right\Vert_{2}\leq \omega_{t},~~t=0, 1, \ldots, N-1,
\end{split}
\end{align}
where $\hat{A}$ and $\hat{U}=(\hat{u}_0,\hat{u}_1,\dots,\hat{u}_{N-1}))$ are taken to be the ideal data, $\tilde{\omega}$ and $ \{\omega_t\}_{t=0}^{N-1}$ are the control levels.

\subsubsection{Influence of noises}\label{subsubsection:4.2.1}
Fixed $\tilde{\omega}=10$ and $\omega_{t}=3, t=0,1,\dots,N-1$, we study the influence of the perturbed noise at the reference outputs on the performance of (\ref{AMOPUL2}). A total of 20 instances of $\{r_t\}_{t=1}^{N}$ with perturbed noises for each $(\mu,\sigma)$ with respect to $\mu=0$, and $\sigma= $ 0.05, 0.1, 0.2, 0.3, 0.4, 0.5, 0.6, 0.7, 0.8 are employed.
The means and standard deviations of the relative error of $A^{\rm a*}$ (RE$A^{\rm a*}$) and approximate cumulative error (ACE) are shown in Table \ref{table:2}, where $(A^{\rm a*}, (u_0^{\rm a*}, u_1^{\rm a*}, \ldots, u_{N-1}^{\rm a*}))$ is an optimal solution of (\ref{AMOPUL2}).
\begin{table}[H]
\centering
\caption{Means and standard deviations of RE$A^{\rm a*}$ and ACE for (\ref{AMOPUL2}) with $\tilde{\omega}=10$ and $\omega_t=3, t=0, 1, \ldots, N-1$.}
\label{table:2}
\begin{tabular}{ccccc}
\hline
\multirow{3}*{($\mu$, $\sigma$)} & \multicolumn{2}{c}{RE$A^{\rm a*}$} & \multicolumn{2}{c}{ACE}\\
 \cmidrule(l){2-3}  \cmidrule(l){4-5}
& mean & std & mean & std\\
\hline
\textbf{(0, 0.05)}& \textbf{5.5770e-11} & \textbf{6.3925e-11} & 3.0853e-01 & 1.3365e-02\\
\textbf{(0, 0.1)} & \textbf{3.4610e-12} & \textbf{1.1336e-12} & 4.7206e-01 & 1.7822e-02\\
\textbf{(0, 0.2)} & \textbf{6.7624e-12} & \textbf{1.2553e-11} & 3.9937e+00 & 3.8810e-01\\
(0, 0.3) & 5.6781e-02 & 6.4317e-03 & \textbf{1.0000e+01} & \textbf{5.2062e-08}\\
(0, 0.4) & 1.6691e-01 & 1.0648e-02 & \textbf{1.0000e+01} & \textbf{5.4277e-08}\\
(0, 0.5) & 2.5858e-01 & 9.7655e-03 & \textbf{1.0000e+01} & \textbf{1.1359e-07}\\
(0, 0.6) & 3.3659e-01 & 1.3048e-02 & \textbf{1.0000e+01} & \textbf{1.4509e-07}\\
(0, 0.7) & 3.9691e-01 & 9.0062e-03 & \textbf{1.0000e+01} & \textbf{1.3490e-07}\\
(0, 0.8) &4.4853e-01 & 1.5996e-02 & \textbf{1.0000e+01} & \textbf{1.8154e-07}\\
\hline
\end{tabular}
\end{table}

Table \ref{table:2} shows that for the reference outputs with small noises (say, $\mu=0, \sigma=0.05$ - $0.2$), the means and standard deviations of RE$A^{\rm a*}$ are all less than $10^{-10}$, while the means of ACE are less than 10 ($=\tilde{\omega}$) with small standard deviations, i.e., the strict inequality ACE~$<\tilde{\omega}$ is satisfied in the second constraint of (\ref{amopul2_special}) , which means that this constraint is inactive.
When the noise becomes larger (say, $\mu=0$, $\sigma=0.3$-$0.8$), the mean of RE$A^{\rm a*}$ increases drastically, while the means of ACE become 10 ($=\tilde{\omega}$) with standard deviations being less than $10^{-6}$, i.e., the equality ACE~$=\tilde{\omega}$ is almost binding in the second constraint of (\ref{amopul2_special}), which means the constraint becomes active.
Therefore, the error of recovering the true $\hat{A}$ mainly comes from the perturbed noises at reference outputs in (\ref{AMOPUL2}) with a fixed control level of $\tilde{\omega}$ and $ \{\omega_{t}\}_{t=0}^{N-1}$. The smaller the perturbed noises are, the higher accuracy is for recovering the true $\hat{A}$.

\subsubsection{Influence of control levels}\label{subsubsection:4.2.2}
Fixed the noise level of $\{r_t\}_{t=1}^N$ with respect to $\mu=0$ and $\sigma=0.5$, we study the influence of control levels $\tilde{\omega}$ and $\{\omega_t\}_{t=0}^{N-1}$ on the performance of (\ref{AMOPUL2}). Set $\tilde{\omega}=$ 2, 10, 20, 30, 40, 50, 60, 70, 80, 90, 100, 110, 120, 130, 140, 150, 160, and $\omega_t=$ 3, 4.5, 6, 8, $t=0, 1,\dots,N-1$, respectively.
The means and standard deviations of RE$A^{\rm a*}$ and ACE are shown in Table \ref{table:3}, where $(A^{\rm a*}, (u_0^{\rm a*}, u_1^{\rm a*}, \ldots, u_{N-1}^{\rm a*}))$ is an optimal solution of (AMOPUL2).
\begin{table}[H]
\centering
\scriptsize
\caption{Means and standard deviations of RE$A^{\rm a*}$ and ACE for (\ref{AMOPUL2}) with the noise level $(\mu=0, \sigma=0.5)$ of $\{r_t\}_{t=1}^N$.}
\label{table:3}
\begin{tabular}{cccccccccc}
\hline
\multirow{3}*{$\left(\{\omega_t\}, \omega\right)$} & \multicolumn{2}{c}{RE$A^{\rm a*}$} & \multicolumn{2}{c}{ACE} &\multirow{3}*{$\left(\{\omega_t\}, \omega\right)$} & \multicolumn{2}{c}{RE$A^{\rm a*}$} & \multicolumn{2}{c}{ACE}\\
 \cmidrule(l){2-3}  \cmidrule(l){4-5} \cmidrule(l){7-8}  \cmidrule(l){9-10}
& mean & std & mean & std & & mean & std & mean & std\\
\hline
(3, 2)& 2.9161e-01 & 9.8335e-03 & 2.0000e+00 & 5.6203e-08 & (6, 2) & 5.6497e-02 & 8.0709e-03 & 2.0000e+00 & 4.1237e-08\\
(3, 10) & 2.5858e-01 & 9.7655e-03 & 1.0000e+01 & 1.1359e-07& (6, 10) & 3.5717e-02 & 7.4733e-03 & 1.0000e+01 & 7.6000e-08\\
(3, 20) & 2.2544e-01 & 9.5567e-03 & 2.0000e+01 &1.2352e-07 & (6, 20) & 1.5907e-02 & 6.7876e-03 & 2.0000e+01 & 6.3577e-08\\
(3, 30) & 1.9594e-01 & 9.4152e-03 & 3.0000e+01 & 1.0370e-07& (6, 30) & 1.8893e-03 & 3.3554e-03 & 2.8972e+01 & 1.3560e+00\\
(3, 40) & 1.6855e-01 & 9.2127e-03 & 4.0000e+01 & 1.2075e-07& (6, 40) & 2.1835e-11 & 9.2773e-11 & 3.3480e+01 & 2.4152e+00\\
(3, 50) & 1.4263e-01 & 9.0004e-03 & 5.0000e+01 & 1.4636e-07& (6, 50) & 2.4596e-13 & 9.5882e-13 & 3.6904e+01 & 2.2087e+00\\
(3, 60) & 1.1790e-01 & 8.7597e-03 & 6.0000e+01 & 1.4685e-07& (6, 60) & 2.1943e-11 & 4.4039e-11 & 4.1768e+01 & 2.1271e+00\\
(3, 70) & 9.4359e-02 & 8.4786e-03 & 7.0000e+01 & 1.6536e-07& (6, 70) & 2.5626e-12 & 3.0437e-12 & 4.5740e+01 & 2.2466e+00\\
(3, 80) & 7.2036e-02 & 8.1467e-03 & 8.0000e+01 & 1.1959e-07& (6, 80) & 2.6641e-12& 3.4618e-12 & 5.1083e+01 & 2.2815e+00\\
(3, 90) & 5.1010e-02 & 7.7407e-03 & 9.0000e+01 & 9.4728e-08& (6, 90) & 5.0958e-12 & 3.5629e-12 & 5.7665e+01 & 2.0483e+00\\
(3, 100) & 3.1406e-02 & 7.2410e-03 & 1.0000e+02 & 1.1437e-07& (6, 100) & 7.8724e-12 & 5.3282e-12 & 6.4275e+01 & 1.8747e+00\\
(3, 110) & 1.3400e-02 & 6.6400e-03 & 1.1000e+02 & 1.1805e-07& (6, 110) & 1.1513e-11& 8.8819e-12 & 7.0651e+01 & 1.7933e+00\\
(3, 120) & 1.2051e-03 & 2.4367e-03 & 1.1852e+02 & 1.6358e+00& (6, 120) & 3.0276e-11& 4.5556e-11 & 7.6993e+01 & 1.6993e+00\\
(3, \textbf{130}) & \textbf{1.7433e-14} & \textbf{3.3912e-14} & \textbf{1.2337e+02} & \textbf{2.1901e+00}& (6, \textbf{130}) & \textbf{3.1542e-11} & \textbf{7.0938e-11} & \textbf{8.2822e+01} & \textbf{1.6659e+00}\\
(3, \textbf{140}) & \textbf{7.7324e-15} & \textbf{4.8385e-15} & \textbf{1.2786e+02} & \textbf{2.1799e+00}& (6, \textbf{140}) & \textbf{1.5461e-11} & \textbf{3.1048e-11} & \textbf{8.7964e+01}& \textbf{1.7284e+00}\\
(3, \textbf{150}) & \textbf{1.2545e-14} & \textbf{5.6427e-15} & \textbf{1.3255e+02} & \textbf{2.1943e+00}& (6, \textbf{150}) & \textbf{5.7891e-12} & \textbf{2.9446e-12} & \textbf{9.2701e+01} & \textbf{1.7949e+00}\\
(3, \textbf{160}) & \textbf{1.2572e-14} & \textbf{6.5987e-15} & \textbf{1.3742e+02} & \textbf{2.1906e+00}& (6, \textbf{160}) & \textbf{8.2313e-13} & \textbf{9.8118e-13} & \textbf{9.6858e+01} & \textbf{1.8979e+00}\\
\hline
(4.5, 2)& 1.5888e-01 & 9.6910e-03 & 2.0000e+00 & 4.2127e-08 & (\textbf{8}, 2) & \textbf{1.5031e-13} & \textbf{2.8242e-13}& \textbf{6.3282e-01} & \textbf{3.6554e-01}\\
(4.5, 10) & 1.3281e-01 & 9.1542e-03 & 1.0000e+01 & 7.4270e-08& (\textbf{8}, 10) & \textbf{5.7621e-13}& \textbf{1.3298e-12}& \textbf{1.7918e+00} & \textbf{5.7976e-01}\\
(4.5, 20) & 1.0664e-01 & 8.7720e-03 & 2.0000e+01 & 6.9067e-08& (\textbf{8}, 20)& \textbf{1.7518e-12} & \textbf{2.7700e-12} & \textbf{5.6229e+00} & \textbf{9.4325e-01}\\
(4.5, 30) & 8.3178e-02 & 8.4139e-03 & 3.0000e+01 & 8.7391e-08& (\textbf{8}, 30) & \textbf{2.5078e-12} & \textbf{4.0871e-12} & \textbf{1.1458e+01}& \textbf{7.7308e-01}\\
(4.5, 40) & 6.1383e-02 & 8.0030e-03 & 4.0000e+01 & 6.1359e-08& (\textbf{8}, 40)& \textbf{1.2160e-12} & \textbf{1.3510e-12}& \textbf{1.8148e+01} & \textbf{1.4540e+00}\\
(4.5, 50) & 4.1028e-02 & 7.5202e-03 & 5.0000e+01 & 1.3419e-07& (\textbf{8}, 50) & \textbf{4.8283e-13} & \textbf{6.0433e-13} & \textbf{2.2691e+01} & \textbf{1.1005e+00}\\
(4.5, 60) & 2.2192e-02 & 6.9524e-03 & 6.0000e+01 & 9.8571e-08& (\textbf{8}, 60) & \textbf{8.7809e-13} & \textbf{3.8889e-12}& \textbf{2.5669e+01} & \textbf{1.0196e+00}\\
(4.5, 70) & 5.8491e-03 & 5.1886e-03 & 6.9716e+01 &6.2190e-01 & (\textbf{8}, 70) & \textbf{9.3531e-12} & \textbf{1.7779e-11}& \textbf{2.7447e+01} & \textbf{1.1129e+00}\\
(4.5, 80) & 3.5200e-06 & 1.5742e-05 & 7.6419e+01 & 2.0065e+00& (\textbf{8}, 80) & \textbf{8.5084e-12} & \textbf{3.9371e-12}& \textbf{2.8603e+01} & \textbf{1.1621e+00}\\
(4.5, 90) & 1.0881e-14 & 1.3872e-14 & 8.1689e+01 & 1.9256e+00& (\textbf{8}, 90) & \textbf{7.3759e-12}& \textbf{8.0317e-12}& \textbf{2.8347e+01} & \textbf{1.2325e+00}\\
(4.5, 100) & 7.6680e-12 & 3.4256e-11 & 8.7126e+01 & 1.8341e+00& (\textbf{8}, 100) & \textbf{3.9721e-12}& \textbf{3.1534e-12} & \textbf{2.7390e+01} & \textbf{1.8718e+00}\\
(4.5, 110) & 7.3171e-13 & 3.2461e-12 & 9.2319e+01 & 1.8332e+00& (\textbf{8}, 110)& \textbf{3.2240e-11}& \textbf{4.2381e-11} & \textbf{2.8903e+01} & \textbf{3.1120e+00}\\
(4.5, 120) & 5.2741e-13 & 2.3420e-12 & 9.7047e+01 & 1.8736e+00& (\textbf{8}, 120) & \textbf{7.3876e-15}& \textbf{9.6040e-15} & \textbf{3.3420e+01} & \textbf{2.7333e+00}\\
(4.5, \textbf{130}) & \textbf{3.7929e-13} & \textbf{1.6870e-12} & \textbf{1.0132e+02} & \textbf{1.9350e+00}& (\textbf{8}, \textbf{130})& \textbf{5.8003e-14}& \textbf{1.0165e-13} & \textbf{4.3559e+01} & \textbf{2.7069e+00}\\
(4.5, \textbf{140}) & \textbf{9.6352e-12} & \textbf{4.0316e-11} & \textbf{1.0527e+02} & \textbf{1.9955e+00}& (\textbf{8}, \textbf{140}) & \textbf{3.6111e-11} & \textbf{7.4397e-11} & \textbf{5.4163e+01} & \textbf{2.5081e+00}\\
(4.5, \textbf{150}) & \textbf{7.8673e-12} & \textbf{3.4276e-11} & \textbf{1.0906e+02} & \textbf{2.0712e+00}& (\textbf{8}, \textbf{150}) & \textbf{2.5740e-14} & \textbf{5.9137e-14} & \textbf{6.2407e+01} & \textbf{1.9772e+00}\\
(4.5, \textbf{160}) & \textbf{5.8232e-11} & \textbf{1.1314e-10} & \textbf{1.1278e+02} & \textbf{2.1251e+00}& (\textbf{8}, \textbf{160}) & \textbf{1.5156e-11} & \textbf{3.5062e-11} & \textbf{6.9692e+01} & \textbf{1.7940e+00}\\
\hline
\end{tabular}
\end{table}
Figures \ref{figure:3} and \ref{figure:4} plot the trends of RE$A^{\rm a*}$ and ACE shown in Table \ref{table:3} with respect to different values of $\tilde{\omega}$ and $\{\omega_t\}_{t=0}^{N-1}$, respectively, in which the length of each error bar above and below the mean value reflects the corresponding standard deviation.
\begin{figure}[H]
\begin{subfigure}[\rm Line plot with error bars of RE$A^{\rm a*}$.]{
\centering
\includegraphics[width=0.45\textwidth]{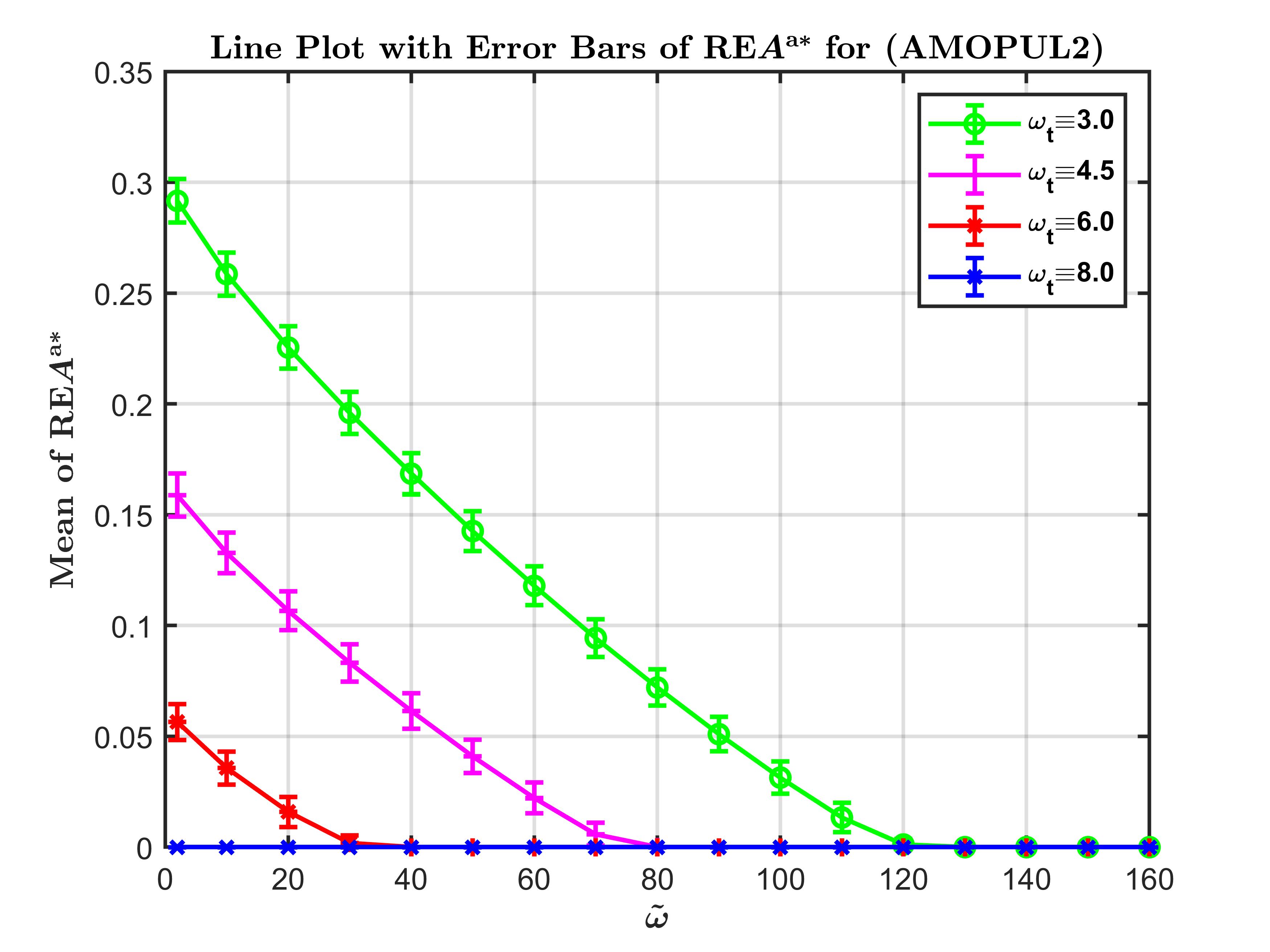}
\label{figure:3}}
\end{subfigure}
\hspace{2.5mm}
\begin{subfigure}[\rm Line plot with error bars of ACE.]{
\centering
\includegraphics[width=0.45\textwidth]{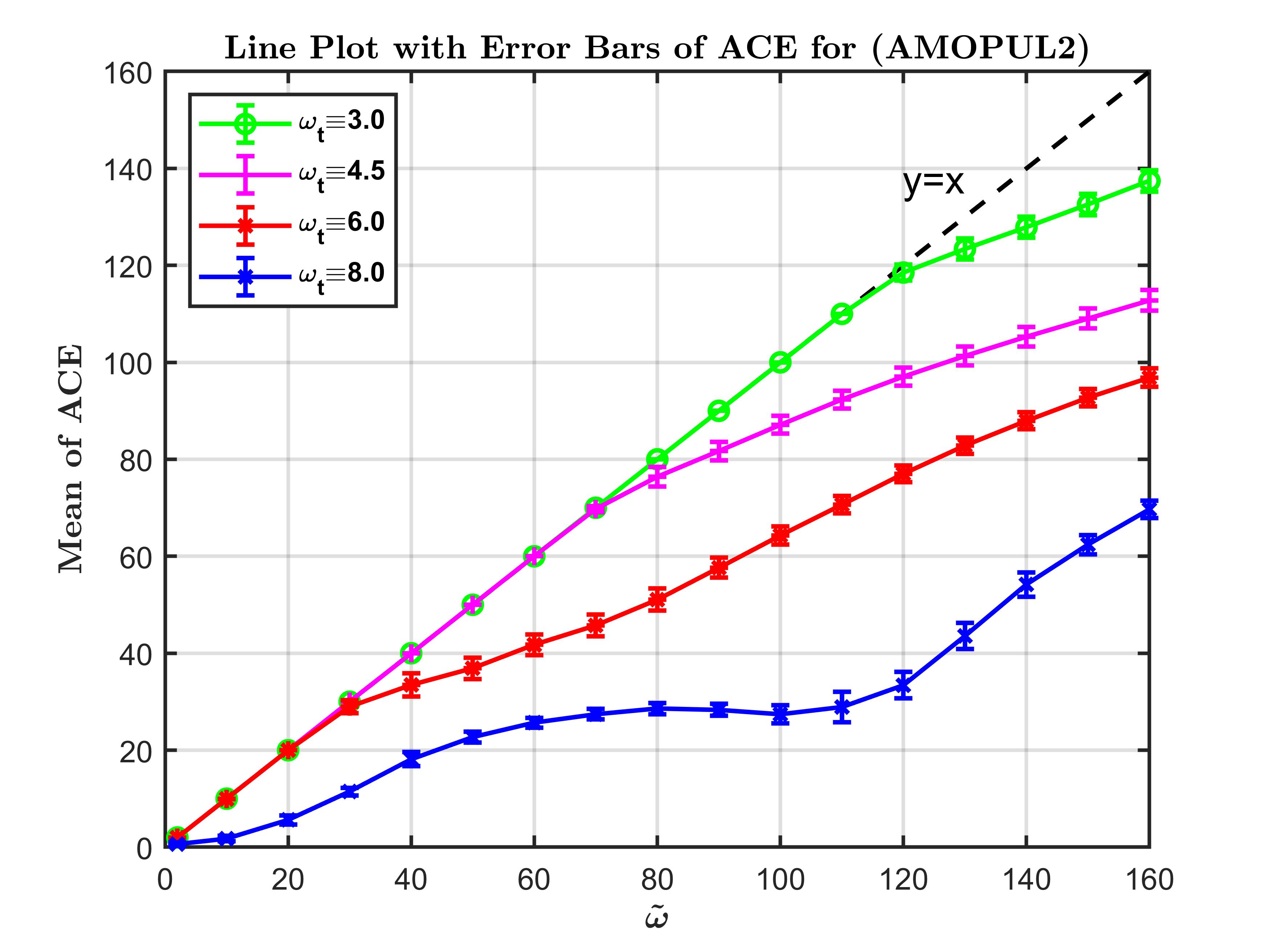}
\label{figure:4}}
\end{subfigure}
\caption{Line plots with error bars of RE$A^{\rm a*}$ and ACE for (\ref{AMOPUL2}) with the noise level $(\mu=0, \sigma=0.5)$ of $\{r_t\}_{t=1}^N$.}
\end{figure}
Two major observations can be made as follows.
\vspace{1mm}

\noindent \textbf{A. Influence of $\boldsymbol{\tilde{\omega}}$}
\vspace{1mm}

For each fixed $\omega_{t}=$ 3, 4.5, 6, $t=0, 1, \ldots, N-1$, the mean of RE$A^{\rm a*}$ decreases to almost zero as $\tilde{\omega}$ increases as shown in Figure \ref{figure:3}. This implies that we can obtain a higher accuracy of recovering the transition matrix through a more relaxed constraint on the cumulative precision in the second constraint of (\ref{amopul2_special}).
However, the means of ACE has an increasing trend with respect to $\tilde{\omega}$ as shown in Figure \ref{figure:4} for each fixed $\omega_t$. Observe that when $\tilde{\omega}$ is small, the ACE curves almost coincide with the line $y=x$, i.e., the equality ACE~$=\tilde{\omega}$ is almost binding in the second constraint of (\ref{amopul2_special}), which means that this constraint becomes active.
When $\tilde{\omega}$ becomes larger, the ACE curves lie below the line $y=x$, i.e., the strict inequality ACE~$<\tilde{\omega}$ is satisfied in the second constraint of (\ref{amopul2_special}), which means that it becomes inactive.
The trade-off between the accuracy of transition matrix recovery and the approximate cumulative error shows that a proper value of $\tilde{\omega}$ plays an important role in the performance of (\ref{AMOPUL2}).
In an extreme case with sufficiently large $\omega_t$ (say, $\omega_{t}=$ 8, $t=0, 1, \ldots, N-1$), since the third constraint of (\ref{amopul2_special}) is sufficiently relaxed, RE$A^{\rm a*}$ is almost equal to zero for each $\tilde{\omega}$ as shown in Figure \ref{figure:3}, while its ACE curve lies below the line $y=x$ as shown in Figure \ref{figure:4}, i.e., the strict inequality ACE~$<\tilde{\omega}$ is satisfied in the second constraint of (\ref{amopul2_special}), which means the constraint becomes inactive.

\vspace{1mm}
\noindent \textbf{B. Influence of $\boldsymbol{\{\omega_{t}\}_{t=0}^{N-1}}$}
\vspace{1mm}

For each fixed value of $\tilde{\omega}$, the means of RE$A^{\rm a*}$ and ACE both decrease as $\omega_{t}$ increases as shown in Figures \ref{figure:3} and \ref{figure:4}, respectively, while we gradually lose the accuracy of recovery of the true $\hat{U}$. In the extreme cases with sufficiently large $\tilde{\omega}=$ 130, 140, 150, 160, i.e., the second constraint of (\ref{amopul2_special}) is much relaxed, RE$A^{\rm a*}$ is almost equal to zero as shown in Figure \ref{figure:3}. In addition, the corresponding curves in Figure \ref{figure:4} lie below the line $y=x$, i.e., the strict inequality ACE~$<\tilde{\omega}$ is satisfied in the second constraint of (\ref{amopul2_special}), which means this constraint  becomes inactive.

Consequently, the values of the control levels $\tilde{\omega}$ and $\{\omega_{t}\}_{t=0}^{N-1}$ play key roles in the performance of (\ref{AMOPUL2}). On one hand, when the values of $\tilde{\omega}$ and $\{\omega_{t}\}_{t=0}^{N-1}$ are small, since the second constraint of (\ref{amopul2_special}) becomes tight, ACE becomes small, i.e., the cumulative precision in the second constraint of (\ref{amopul2_special}) becomes high. However, the value of RE$A^{\rm a*}$ becomes large, i.e., the recovery precision of the true $\hat{A}$ becomes low. In the extreme cases with sufficiently small $\tilde{\omega}$ and $\{\omega_t\}_{t=0}^{N-1}$, i.e., extremely high requirements of the cumulative precision in the second constraint of (\ref{amopul2_special}) and the recovery precision of the true $(\hat{u}_0, \hat{u}_1, \ldots, \hat{u}_{N-1})$ in the third constraint of (\ref{amopul2_special}), (\ref{AMOPUL2}) may be infeasible, for example, setting $\tilde{\omega}=\omega_t=0, t=0, 1, \ldots, N-1$.
On the other hand, when the values of $\tilde{\omega}$ and $\{\omega_{t}\}_{t=0}^{N-1}$ become sufficiently large, the second and third constraints of (\ref{amopul2_special}) may become inactive, while RE$A^{\rm a*}$ becomes small, i.e., the recovery precision of the true $\hat{A}$ becomes high.

Based on the results for (\ref{AMOPUL1}) and (\ref{AMOPUL2}) in Subsection \ref{subsection:4.1} and Subsubsection \ref{subsubsection:4.2.1}, we can see that the accuracy of given reference outputs $\{r_t\}_{t=1}^N$ is the key for the performance of (\ref{AMOPUL}). More accurate $\{r_t\}_{t=1}^N$ leads to better performance of (\ref{AMOPUL}).

In addition, the results of (\ref{AMOPUL2}) in Subsubsection \ref{subsubsection:4.2.2} indicate that the control level of approximate cumulative error plays an important
role in the performance of (\ref{AMOPUL}). With a proper value of the control level, the proposed SDP approximation model may perform very well.

\section{Concluding remarks}\label{section:5}
This paper studies a matrix optimization problem over an uncertain linear system on finite horizon, in which the uncertain transition matrix is regarded as a decision variable. To decouple the entanglement of decision variables caused by the corresponding multivariate polynomial constraints for computational efficiency, we construct a polynomial-time solvable SDP approximation model by taking the given reference values as system outputs at each stage. Theoretical and numerical results show that the reference values of outputs and control levels play key roles of the proposed approach. The SDP approximation performs very well when the noises of reference outputs are small and control levels are proper.

One potential extension of this work is to treat other parameter matrices of an uncertain linear system as decision variables. The entangled decision variables can be similarly decoupled by using the given reference values as substitutions to construct an SDP approximation model. Another potential extension is to study the matrix optimization problem over an uncertain non-linear system. A possible way is to approximate the uncertain nonlinear system by a linear system using the cutting plane/simplicial decomposition methods \cite{Bertsekas2015}. And then develop a polynomial-time solvable SDP approximation model.

\end{document}